\documentclass[fleqn]{article}

\usepackage{amsmath,amssymb,amsthm,url,graphics}
\usepackage{tikz}
\usetikzlibrary {arrows.meta}

\title{Lyndon interpolation property for extensions of $\mathbf{S4}$ and intermediate propositional logics}
\author{Taishi Kurahashi\footnote{Email: kurahashi@people.kobe-u.ac.jp}
\footnote{Graduate School of System Informatics,
Kobe University,
1-1 Rokkodai, Nada, Kobe 657-8501, Japan.}
}
\date{}

\theoremstyle{plain}
\newtheorem{thm}{Theorem}[section]
\newtheorem{lem}[thm]{Lemma}
\newtheorem{prop}[thm]{Proposition}
\newtheorem{cor}[thm]{Corollary}
\newtheorem{fact}[thm]{Fact}

\theoremstyle{definition}
\newtheorem{defn}[thm]{Definition}

\newtheorem{rem}[thm]{Remark}

\newcommand{\Triv}{\mathbf{Triv}}
\newcommand{\SFi}{\mathbf{S5}}
\newcommand{\SF}{\mathbf{S4}}
\newcommand{\K}{\mathbf{K}}
\newcommand{\GV}{\mathbf{GV}}

\newcommand{\LP}{\mathbf{LP}}
\newcommand{\LC}{\mathbf{LC}}
\newcommand{\LS}{\mathbf{LS}}
\newcommand{\LV}{\mathbf{LV}}
\newcommand{\KC}{\mathbf{KC}}
\newcommand{\Cl}{\mathbf{Cl}}

\newcommand{\xr}[1]{\xrightarrow[#1]{(P^+, P^-)}}
\newcommand{\Fm}[1]{\Phi^{(P^+, P^-)}_{#1}}

\newcommand{\Cj}[1]{\chi^{(P^+, P^-)}_{#1}}
\newcommand{\Cjc}[1]{\chi^{(P^-, P^+)}_{#1}}

\begin{document}

\maketitle

\begin{abstract}
We study the Lyndon interpolation property (LIP) and the uniform Lyndon interpolation property (ULIP) for extensions of $\mathbf{S4}$ and intermediate propositional logics. 
We prove that among the 18 consistent normal modal logics of finite height extending $\mathbf{S4}$ known to have CIP, 11 logics have LIP and 7 logics do not. 
We also prove that for intermediate propositional logics, the Craig interpolation property, LIP, and ULIP are equivalent. 
\end{abstract}

\section{Introduction}

The Craig Interpolation Theorem was first proved by Craig \cite{Craig57} for classical first-order predicate logic, and has been studied for many logics. 
Let $v(\varphi)$ denote the set of all propositional variables contained in a propositional formula $\varphi$. 
We say that a propositional logic $L$ has the \emph{Craig interpolation property} (CIP) iff for any formulas $\varphi$ and $\psi$, if $L \vdash \varphi \to \psi$, then there exists a formula $\rho$ such that $v(\rho) \subseteq v(\varphi) \cap v(\psi)$, $L \vdash \varphi \to \rho$, and $L \vdash \rho \to \psi$. 
CIP has been extensively investigated in modal and intermediate propositional logics. 
In particular, Maksimova \cite{Maksimova77} proved that there are exactly seven consistent intermediate propositional logics having CIP. 
They are classical propositional logic $\Cl$, $\LS$, $\LV$, $\LP_2$, the G\"odel--Dummett logic $\LC$, the Jankov logic $\KC$, and intuitionistic propositional logic $\mathbf{Int}$\footnote{For the names of these logics, we adopt those presented in Gabbay and Maksimova's book \cite{GM05}.}. 
Moreover, Maksimova proved that there are at most 36 consistent normal extensions of $\SF$ having CIP, and that at least 30 consistent logics among them have CIP (see \cite{GM05,Maksimova91}). 
In particular, there are exactly 18 consistent normal modal logics of finite height extending $\SF$ and having CIP (see Table \ref{table:IP}\footnote{See Gabbay and Maksimova \cite{GM05} for the detailed definitions of the logics included in the table}). 

Lyndon \cite{Lyndon59} proved that classical first-order predicate logic enjoys a stronger interpolation property. 
Let $v^+(\varphi)$ (resp.~$v^-(\varphi)$) be the set of all propositional variables occurring in a propositional formula $\varphi$ positively (resp.~negatively). 
We say that a propositional logic $L$ has the \emph{Lyndon interpolation property} (LIP) iff for any formulas $\varphi$ and $\psi$, if $L \vdash \varphi \to \psi$, then there exists a formula $\rho$ such that $v^\circ(\rho) \subseteq v^\circ(\varphi) \cap v^\circ(\psi)$ for $\circ \in \{+, -\}$, $L \vdash \varphi \to \rho$, and $L \vdash \rho \to \psi$. 
LIP for extensions of $\SF$ has been studied by Maksimova \cite{Maksimova82, Maksimova14}, Fitting \cite{Fitting83}, and Kuznets \cite{Kuznets16, Kuznets18}. 
It is also verifiable that Shimura's proofs of CIP for some extensions of $\SF$ also work with respect to LIP \cite{Shimura92}\footnote{Shimura proved that CIP of $\mathbf{S4.4}$, $\mathbf{GW}$, and $\Gamma(\LP_2, \omega, 1)$ follows from that of $\SFi$, $\Triv$, and $\SFi$, respectively. It is easily shown that his proofs also show that LIP of $\mathbf{S4.4}$, $\mathbf{GW}$, and $\Gamma(\LP_2, \omega, 1)$ follows from that of $\SFi$, $\Triv$, and $\SFi$, respectively. The first explicit proof of LIP for $\mathbf{GW}$ was presented in Maksimova \cite{Maksimova14}.}.
Among the 30 logics known to have CIP, it has been proved that 12 logics actually have LIP \footnote{These 12 logics are $\mathbf{Triv}$ \cite{Maksimova82}, $\SFi$ \cite{Fitting83}, $\mathbf{GW.2}$ \cite{Maksimova82}, $\mathbf{S4.4}$ \cite{Shimura92}, $\mathbf{GW}$ \cite{Shimura92,Maksimova14}, $\Gamma(\LP_2, \omega, 1)$ \cite{Shimura92}, $\mathbf{S4.2}$ \cite{Kuznets16}, $\mathbf{S4.1.2}$ \cite{Maksimova82}, $\mathbf{Grz.2}$ \cite{Maksimova14}, $\SF$ \cite{Fitting83,Maksimova82}, $\mathbf{S4.1}$ \cite{Maksimova82}, and $\mathbf{Grz}$ \cite{Maksimova14}.} and 4 logics do not.
LIP for intermediate propositional logics has been studied by Maksimova \cite{Maksimova82, Maksimova14} and Kuznets and Lellmann \cite{KL18,KL21}. 
Among these 7 logics known to have CIP, LIP for the six logics other than $\LV$ has already been proved\footnote{Maksimova \cite{Maksimova14} explicitly mentioned that LIP of $\LP_2$ immediately follows from that of $\mathbf{GW}$. 
Even earlier, although implicitly, Shimura's results also yield the LIP of $\LP_2$.} 
and the problem of LIP for $\LV$ was open (see Maksimova \cite{Maksimova14} and Table \ref{table:IP2}).

\begin{table}[ht]

 \centering
  \begin{tabular}{|l||l|l|}
   \hline
   Logic & CIP & LIP \\
   \hline
   \hline
    $\Gamma(\Cl, 1, 0)$, $\Triv$ & Maksimova \cite{Maksimova80} & Maksimova \cite{Maksimova82} \\
   \hline
    $\Gamma(\Cl, 2, 0)$ & Maksimova \cite{Maksimova80} & $\times$ Maksimova \cite{Maksimova82} \\
   \hline
    $\Gamma(\Cl, \omega, 0)$, $\SFi$ & Schumm \cite{Schumm76} & Fitting \cite{Fitting83} \\
   \hline 
   \hline
    $\Gamma(\LS, 1, 1)$, $\mathbf{GW.2}$ & Schumm \cite{Schumm76} & Maksimova \cite{Maksimova82} \\
   \hline
    $\Gamma(\LS, 1, 2)$ & Maksimova \cite{Maksimova80} & $\times$ Maksimova \cite{Maksimova82} \\
   \hline
    $\Gamma(\LS, 1, \omega)$ & Maksimova \cite{Maksimova80} & $\times$ This paper (Theorem \ref{failure}) \\
   \hline
    $\Gamma(\LS, 2, 1)$ & Maksimova \cite{Maksimova80} & This paper (Theorem \ref{thm:LS,2,1}) \\
   \hline
    $\Gamma(\LS, \omega, 1)$, $\mathbf{S4.4}$ & Schumm \cite{Schumm76} & Shimura \cite{Shimura92} \\
   \hline 
   \hline
    $\Gamma(\LV, 1, 1)$, $\mathbf{GV}$ & Maksimova \cite{Maksimova80} & This paper (Theorem \ref{thm:GV}) \\
   \hline
    $\Gamma(\LV, 1, 2)$ & Maksimova \cite{Maksimova80} & $\times$ Maksimova \cite{Maksimova82} \\
   \hline
    $\Gamma(\LV, 1, \omega)$ & Maksimova \cite{Maksimova80} & $\times$ This paper (Theorem \ref{failure}) \\
   \hline
    $\Gamma(\LV, 2, 1)$ & Maksimova \cite{Maksimova80} & This paper (Theorem \ref{thm:LV,2,1}) \\
   \hline
    $\Gamma(\LV, \omega, 1)$ & Maksimova \cite{Maksimova80} & This paper (Theorem \ref{thm:LV,o,1}) \\
   \hline
   \hline
    $\Gamma(\LP_2, 1, 1)$, $\mathbf{GW}$ & Schumm \cite{Schumm76} & Shimura \cite{Shimura92}; Maksimova \cite{Maksimova14} \\
   \hline
    $\Gamma(\LP_2, 1, 2)$ & Maksimova \cite{Maksimova80} & $\times$ Maksimova \cite{Maksimova82} \\
   \hline
    $\Gamma(\LP_2, 1, \omega)$ & Maksimova \cite{Maksimova80} & $\times$ This paper (Theorem \ref{failure}) \\
   \hline
    $\Gamma(\LP_2, 2, 1)$ & Maksimova \cite{Maksimova80} & This paper (Theorem \ref{thm:LP2,2,1}) \\
   \hline
    $\Gamma(\LP_2, \omega, 1)$ & Schumm \cite{Schumm76} & Shimura \cite{Shimura92} \\
   \hline
   \end{tabular}
 \caption{CIP and LIP for consistent normal modal logics of finite height extending $\SF$}
 \label{table:IP}
 \end{table}

In the present paper, we focus on the 18 extensions of $\SF$ of finite height listed in Table \ref{table:IP}. 
The first goal of this paper is to provide a complete description concerning LIP for these logics.
We prove that five of them have LIP and three do not. 
Then, we conclude that among these 18 logics, 11 logics have LIP and 7 logics do not (see Table \ref{table:IP}). 
Our strategy for proving LIP for these logics is to prove a stronger property, the uniform Lyndon interpolation property. 
We say that a propositional logic $L$ has the \emph{uniform Lyndon interpolation property} (ULIP) iff for any formula $\varphi$ and any finite sets $P^+, P^-$ of propositional variables, there exists a formula $\theta$ satisfying the following three conditions:
\begin{enumerate}
    \item $L \vdash \varphi \to \theta$, 
    \item $v^\circ(\theta) \subseteq v^\circ(\varphi) \setminus P^\circ$ for $\circ \in \{+, -\}$,
    \item for any formula $\psi$ with $L \vdash \varphi \to \psi$ and $v^\circ(\psi) \cap P^\circ = \emptyset$ for $\circ \in \{+, -\}$, we have $L \vdash \theta \to \psi$. 
\end{enumerate}
Here, $\theta$ is called a \emph{uniform Lyndon interpolant} of $(\varphi, P^+, P^-)$ in $L$. 
ULIP is a simultaneous strengthening of LIP and the uniform interpolation property (UIP), and the notion ULIP was introduced in \cite{Kurahashi20}. 
In that paper, by extending the semantic technique for proving UIP developed by Visser \cite{Visser96}, it was proved that the modal logics such as $\K$, $\mathbf{KT}$, $\mathbf{KB}$, $\mathbf{GL}$, and $\mathbf{Grz}$ have ULIP.
We say that a logic $L$ is \emph{locally tabular} iff for any finite set $P$ of propositional variables, there are finitely many formulas built from variables in $P$ up to $L$-provable equivalence. 
It is known that LIP and ULIP are equivalent for every locally tabular logic (cf.~\cite{Kurahashi20}).
Since the logics in Table \ref{table:IP} are all locally tabular, to prove LIP for these logics, we will adopt the method of proving stronger ULIP.

We also discuss LIP and ULIP for intermediate propositional logics. 
The problem of LIP for the modal logic $\GV$ was stated to be open in Maksimova \cite{Maksimova14}, but in this paper, by proving LIP for $\GV$, we immediately obtain LIP for the intermediate logic $\LV$. 
From this, we conclude that CIP and LIP are equivalent for intermediate propositional logics.
Since the logics $\Cl$, $\LS$, $\LV$, $\LP_2$, and $\LC$ are known to be locally tabular (cf.~\cite[p.~428]{CZ97}), they are shown to enjoy ULIP. 
ULIP for $\mathbf{Int}$ follows immediately from ULIP for $\mathbf{Grz}$.
In the present paper, we prove ULIP for $\mathbf{Grz.2}$, which implies that $\mathbf{KC}$ also has ULIP.
From these investigations, we also conclude that CIP and ULIP are equivalent for intermediate propositional logics (see Table \ref{table:IP2}). 

\begin{table}[ht]
 \centering
\scriptsize{
\begin{tabular}{|l||l|l|l|l|}
\hline
Logic & CIP & LIP & UIP & ULIP \\
\hline \hline
$\Cl$ & Craig \cite{Craig57} & Lyndon \cite{Lyndon59} & $\checkmark$ & $\checkmark$ \\
\hline
$\LS$ & Maksimova \cite{Maksimova77} & Maksimova \cite{Maksimova82} & $\checkmark$ &  $\checkmark$ \\
\hline
$\LV$ & Maksimova \cite{Maksimova77} & \begin{tabular}{l} This paper \\ (Corollary \ref{cor1}) \end{tabular} & $\checkmark$ & $\checkmark$ \\
\hline
$\LP_2$ & Maksimova \cite{Maksimova77} & \begin{tabular}{l} Shimura \cite{Shimura92} \\ Maksimova \cite{Maksimova14} \end{tabular} & $\checkmark$ & $\checkmark$ \\
\hline
$\LC$ & Maksimova \cite{Maksimova77} & \begin{tabular}{l} Kuznets and \\ Lellmann \cite{KL18}\end{tabular} & $\checkmark$ & $\checkmark$ \\
\hline
$\KC$ & Gabbay \cite{Gabbay71} & Maksimova \cite{Maksimova82} & Maksimova \cite{Maksimova14} & \begin{tabular}{l} This paper \\ (Corollary \ref{cor3}) \end{tabular} \\
\hline
$\mathbf{Int}$ & Sch\"utte \cite{Schutte62} & Maksimova \cite{Maksimova82} & Pitts \cite{Pitts92} & \begin{tabular}{l} This paper \\ (Corollary \ref{cor2}) \end{tabular} \\
\hline
\end{tabular}
} \caption{Interpolation properties for intermediate propositional logics}
 \label{table:IP2}
\end{table}

This paper is organized as follows. 
In Section \ref{sec:Cl_ULIP}, we develop the basis of our method of proving ULIP for several logics. 
We present a simple proof of ULIP for classical propositional logic $\Cl$ using the basis of our method. 
Section \ref{sec:ULIP} is devoted to developing a sufficient condition for modal logics to have ULIP. 
More precisely, we introduce the notion that a class of Kripke models enjoys $n$-IP and prove that a logic $L$ has ULIP if $L$ is sound and complete with respect to a class of models enjoying $n$-IP for some natural number $n$. 
In Section \ref{sec:lemmas}, we prepare some lemmas on matching between clusters of Kripke models, which are used in our proofs of ULIP for several logics. 
In Sections \ref{sec:Cl}, \ref{sec:LS}, \ref{sec:LV}, and \ref{sec:LP2}, we use our method to prove ULIP for 11 logics in Table 1, including 6 logics for which LIP has already been proved.
In Section \ref{sec:failure}, we discuss the failure of LIP for some logics.
At last, in Section \ref{sec:intermediate}, we discuss LIP and ULIP for intermediate propositional logics.

\section{ULIP for classical propositional logic}\label{sec:Cl_ULIP}

In this section, we develop the basis of our method for proving ULIP of several modal logics and we apply this basis to prove ULIP of classical propositional logic. 
For this section only, we assume that formulas mean formulas of classical propositional logic. 
The language of classical propositional logic consists of propositional variables, logical constant $\bot$, and logical connectives $\land, \lor, \neg$ and $\to$. 

For each formula $\varphi$, we define the sets $v^+(\varphi)$ and $v^-(\varphi)$ recursively as follows: 
\begin{itemize}
\item $v^+(p) = \{p\}$ and $v^-(p) = \emptyset$ for every propositional variable $p$, 
\item $v^\circ(\bot) = \emptyset$ for $\circ \in \{+, -\}$, 
\item $v^\circ(\varphi \ast \psi) = v^\circ(\varphi) \cup v^\circ(\psi)$ for $\circ \in \{+, -\}$ and $\ast \in \{\land, \lor\}$, 
\item $v^+(\neg \varphi) = v^-(\varphi)$ and $v^-(\neg \varphi) = v^+(\varphi)$, 
\item $v^+ (\varphi \to \psi) = v^-(\varphi) \cup v^+(\psi)$ and $v^-(\varphi \to \psi) = v^+(\varphi) \cup v^-(\psi)$.
\end{itemize}
Also, let $v(\varphi) : = v^+(\varphi) \cup v^-(\varphi)$. 

\begin{defn}\label{defn:basic}
    Let $P^+$ and $P^-$ be any finite sets of propositional variables. 
\begin{enumerate}
    \item A formula $\varphi$ is said to be a \emph{$(P^+, P^-)$-formula} iff $v^\circ(\varphi) \subseteq P^\circ$ for $\circ \in \{+, -\}$. 

    \item Let $\Fm{0}$ denote a fixed finite set of $(P^+, P^-)$-formulas such that for any $(P^+, P^-)$-formula $\varphi$, there exists a $\psi \in \Fm{0}$ such that $\Cl \vdash \varphi \leftrightarrow \psi$. 
    The existence of such a finite set $\Fm{0}$ is easily proved. 
    Here, the subscript $0$ indicates box-free. 
    
    \item For every truth assignments $\mathsf{V}_0$ and $\mathsf{V}_1$ of formulas, we write
    \[
        \mathsf{V}_0 \xr{0} \mathsf{V}_1
    \]
    iff for any $(P^+, P^-)$-formula $\varphi$, if $\mathsf{V}_0(\varphi) = 1$, then $\mathsf{V}_1(\varphi) = 1$. 

    \item For every truth assignment $\mathsf{V}$, let $\Cj{0}[\mathsf{V}]$ denote the $(P^+, P^-)$-formula
    \[
        \bigwedge \{\varphi \in \Fm{0} \mid \mathsf{V}(\varphi) = 1\}.
    \]
    \end{enumerate}
\end{defn}

The following lemma is easily verified, and therefore we use the lemma freely without referring to it.

\begin{lem}[Cf.~{\cite[Proposition 6]{Kurahashi20}}]\label{lem:equiv}
Let $P^+$ and $P^-$ be any finite sets of propositional variables and $\mathsf{V}_0$ and $\mathsf{V}_1$ be any truth assignments. 
Then, the following are equivalent: 
\begin{enumerate}
    \item $\mathsf{V}_0 \xr{0} \mathsf{V}_1$. 
    \item $\mathsf{V}_1 \xrightarrow[0]{(P^-, P^+)} \mathsf{V}_0$. 
    \item $\mathsf{V}_1 \bigl(\Cj{0}[\mathsf{V}_0] \bigr) = 1$. 
    \item $\mathsf{V}_0 \bigl(\Cjc{0}[\mathsf{V}_1] \bigr) = 1$. 
\end{enumerate}
\end{lem}

We prove the following theorem, which is an adaptation of a result proved in \cite[Proof of Lemma 1]{Kurahashi20} to the framework of this section.

\begin{thm}\label{basis}
Let $P_0^+, P_1^+, P_2^+, P_0^-, P_1^-$ and $P_2^-$ be finite sets of propositional variables such that $P_0^\circ$, $P_1^\circ$ and $P_2^\circ$ are pairwise disjoint for $\circ \in \{+, -\}$. 
Let $\mathsf{V}_0$ and $\mathsf{V}_1$ be truth assignments. 
If $\mathsf{V}_0 \xrightarrow[0]{(P_1^+, P_1^-)} \mathsf{V}_1$, then there exists a truth assignment $\mathsf{V}^*$ such that $\mathsf{V}_0 \xrightarrow[0]{(P_0^+ \cup P_1^+, P_0^- \cup P_1^-)} \mathsf{V}^*$ and $\mathsf{V}^* \xrightarrow[0]{(P_1^+ \cup P_2^+, P_1^- \cup P_2^-)} \mathsf{V}_1$. 
\end{thm}
\begin{proof}
    Let $\mathsf{V}_0$ and $\mathsf{V}_1$ be any truth assignments such that $\mathsf{V}_0 \xrightarrow[0]{(P_1^+, P_1^-)} \mathsf{V}_1$. 
    We define a truth assignment $\mathsf{V}^*$ by referring to Table \ref{table:Assignment}. 
    For each propositional variable $p$, let $\mathsf{V}^*(p) = 1$ iff $p$ meets one of the conditions stated in the 16 rows of the table. 
    For example, the fifth row in the table states the condition that $p \in P_1^+ \cap P_0^-$, $p \notin P_0^+ \cup P_2^+ \cup P_1^- \cup P_2^-$, and $\mathsf{V}_0(p) = 1$. 

\begin{table}[ht]

 \centering
  \begin{tabular}{|c||c|c|c|c|c|c|l|}
   \hline
    & $P_0^+$ & $P_1^+$ & $P_2^+$ & $P_0^-$ & $P_1^-$ & $P_2^-$ & \\
   \hline \hline
    1 & $\checkmark$ & & & $\checkmark$ & & & $\mathsf{V}_0(p) = 1$ \\
    \hline
    2 & $\checkmark$ & & & & $\checkmark$ & & $\mathsf{V}_0(p) = 1$ \\
    \hline
    3 & $\checkmark$ & & & & & $\checkmark$ & $\mathsf{V}_0(p) = 1$ or $\mathsf{V}_1(p) = 1$ \\
    \hline
    4 & $\checkmark$ & & & & & & $\mathsf{V}_0(p) = 1$ \\
    \hline
    5 & & $\checkmark$ & & $\checkmark$ & & & $\mathsf{V}_0(p) = 1$ \\
    \hline
    6 & & $\checkmark$ & & & $\checkmark$ & & $\mathsf{V}_1(p) = 1$ \\
    \hline
    7 & & $\checkmark$ & & & & $\checkmark$ & $\mathsf{V}_1(p) = 1$ \\
    \hline
    8 & & $\checkmark$ & & & & & $\mathsf{V}_0(p) = 1$ \\
    \hline
    9 & & & $\checkmark$ & $\checkmark$ & & & $\mathsf{V}_0(p) = 1$ and $\mathsf{V}_1(p) = 1$ \\
    \hline
    10 & & & $\checkmark$ & & $\checkmark$ & & $\mathsf{V}_1(p) = 1$ \\
    \hline
    11 & & & $\checkmark$ & & & $\checkmark$ & $\mathsf{V}_1(p) = 1$ \\
    \hline
    12 & & & $\checkmark$ & & & & $\mathsf{V}_1(p) = 1$ \\
    \hline
    13 & & & & $\checkmark$ & & & $\mathsf{V}_0(p) = 1$ \\
    \hline
    14 & & & & & $\checkmark$ & & $\mathsf{V}_1(p) = 1$ \\
    \hline
    15 & & & & & & $\checkmark$ & $\mathsf{V}_1(p) = 1$ \\
    \hline
    16 & & & & & & & $\mathsf{V}_0(p) = 1$ \\
    \hline
\end{tabular}
\caption{The definition of $\mathsf{V}^*$}
 \label{table:Assignment}
\end{table}

For the proof of $\mathsf{V}_0 \xrightarrow[0]{(P_0^+ \cup P_1^+, P_0^- \cup P_1^-)} \mathsf{V}^*$, we simultaneously prove the following two conditions by induction on the construction of $\varphi$: 
\begin{enumerate}
    \item If $\varphi$ is a $(P_0^+ \cup P_1^+, P_0^- \cup P_1^-)$-formula and $\mathsf{V}_0(\varphi) = 1$, then $\mathsf{V}^*(\varphi) = 1$. 

     \item If $\varphi$ is a $(P_0^- \cup P_1^-, P_0^+ \cup P_1^+)$-formula and $\mathsf{V}_0(\varphi) = 0$, then $\mathsf{V}^*(\varphi) = 0$.  
\end{enumerate}

The case of $\bot$ trivially holds. 
We prove the case of a propositional variable $p$. 

\medskip

1. Suppose $p \in P_0^+ \cup P_1^+$ and $\mathsf{V}_0(p) = 1$. 
If $p \in P_0^+$, then $p$ meets one of 1, 2, 3, and 4, and hence $\mathsf{V}^*(p) = 1$. 
If $p \in P_1^+$, then we have $\mathsf{V}_1(p) = 1$ because $\mathsf{V}_0 \xrightarrow[0]{(P_1^+, P_1^-)} \mathsf{V}_1$. 
Then, $p$ meets one of 5, 6, 7, and 8, and thus we obtain $\mathsf{V}^*(p) = 1$. 

\medskip

2. Suppose $p \in P_0^- \cup P_1^-$ and $\mathsf{V}^*(p) = 1$. 
In this case, $p$ meets one of 1, 2, 5, 6, 9, 10, 13, and 14. 
If $p$ meets one of 1, 2, 5, 9, and 13, then we have $\mathsf{V}_0(p) = 1$. 
If $p$ meets one of 6, 10, and 14, then $p \in P_1^-$ and $\mathsf{V}_1(p) = 1$. 
Since $\mathsf{V}_0 \xrightarrow[0]{(P_1^+, P_1^-)} \mathsf{V}_1$, we obtain $\mathsf{V}_0(p) = 1$. 

\medskip

The cases of the boolean combinations are easily proved by using the induction hypothesis. 

\medskip

For the proof of $\mathsf{V}^* \xrightarrow[0]{(P_1^+ \cup P_2^+, P_1^- \cup P_2^-)} \mathsf{V}_1$, we simultaneously prove the following two conditions by induction on the construction of $\varphi$: 
\begin{enumerate}
    \item If $\varphi$ is a $(P_1^+ \cup P_2^+, P_1^- \cup P_2^-)$-formula and $\mathsf{V}^*(\varphi) = 1$, then $\mathsf{V}_1(\varphi) = 1$. 

     \item If $\varphi$ is a $(P_1^- \cup P_2^-, P_1^+ \cup P_2^+)$-formula and $\mathsf{V}^*(\varphi) = 0$, then $\mathsf{V}_1(\varphi) = 0$.  
\end{enumerate}

We give only the proof of the case that $\varphi$ is a propositional variable $p$. 

\medskip

1. Suppose $p \in P_1^+ \cup P_2^+$ and $\mathsf{V}^*(p) = 1$. 
Then, $p$ meets one of 5, 6, 7, 8, 9, 10, 11, and 12. 
If $p$ meets one of 6, 7, 9, 10, 11, and 12, then $\mathsf{V}_1(p)=1$. 
If $p$ meets one of 5 and 8, then $p \in P_1^+$ and $\mathsf{V}_0(p) = 1$. 
Since $\mathsf{V}_0 \xrightarrow[0]{(P_1^+, P_1^-)} \mathsf{V}_1$, we obtain $\mathsf{V}_1(p) = 1$. 

\medskip

2. Suppose $p \in P_1^- \cup P_2^-$ and $\mathsf{V}_1(p) = 1$. 
If $p \notin P_0^+ \cap P_1^-$, then $p$ meets one of 3, 6, 7, 10, 11, 14, and 15, and hence $\mathsf{V}^*(p) = 1$. 
If $p \in P_0^+ \cap P_1^-$, then we have $\mathsf{V}_0(p) = 1$ because $\mathsf{V}_0 \xrightarrow[0]{(P_1^+, P_1^-)} \mathsf{V}_1$. 
Then $p$ meets 2, and thus we get $\mathsf{V}^*(p) = 1$. 
\end{proof}

By using Theorem \ref{basis}, we can easily prove ULIP for classical propositional logic $\Cl$. 

\begin{thm}\label{Cl_ULIP}
The logic $\Cl$ has ULIP. 
\end{thm}
\begin{proof}
Let $\varphi$ be any formula and $P^+$ and $P^-$ be any finite sets of propositional variables. 
Let $P_0^\circ : = P^\circ$ and $P_1^\circ : = v^\circ(\varphi) \setminus P^\circ$ for $\circ \in \{+, -\}$. 
Let $\theta$ be the $(P_1^+, P_1^-)$-formula
\[
    \bigwedge \{\psi \in \Phi_0^{(P_1^+, P_1^-)} \mid \Cl \vdash \varphi \to \psi\}.
\]
We prove that $\theta$ is a uniform Lyndon interpolant of $(\varphi, P^+, P^-)$ in $\Cl$. 
The conditions $\Cl \vdash \varphi \to \theta$ and $v^\circ(\theta) \subseteq v^\circ(\varphi) \setminus P^\circ$ for $\circ \in \{+, -\}$ are easily verified. 
Let $\rho$ be any formula such that $\Cl \nvdash \theta \to \rho$ and $v^\circ(\rho) \cap P^\circ = \emptyset$ for $\circ \in \{+, -\}$. 
It suffices to prove $\Cl \nvdash \varphi \to \rho$. 
Let $P_2^\circ : = v^\circ(\rho) \setminus v^\circ(\varphi)$ for $\circ \in \{+, -\}$. 
Since $v^\circ(\rho) \cap P^\circ = \emptyset$, we have that $P_0^\circ$, $P_1^\circ$, and $P_2^\circ$ are pairwise disjoint for $\circ \in \{+, -\}$. 

Since $\Cl \nvdash \theta \to \rho$, there exists a truth assignment $\mathsf{V}_1$ such that $\mathsf{V}_1(\theta) = 1$ and $\mathsf{V}_1(\rho) = 0$. 
Since $\mathsf{V}_1\bigl(\chi_0^{(P_1^-, P_1^+)}[\mathsf{V}_1] \bigr) = 1$, we have $\mathsf{V}_1\bigl(\theta \to \neg \chi_0^{(P_1^-, P_1^+)}[\mathsf{V}_1] \bigr) = 0$. 
Thus, $\Cl \nvdash \theta \to \neg \chi_0^{(P_1^-, P_1^+)}[\mathsf{V}_1]$. 
Since $\neg \chi_0^{(P_1^-, P_1^+)}[\mathsf{V}_1]$ is a $(P_1^+, P_1^-)$-formula, by the definition of $\theta$, we obtain $\Cl \nvdash \varphi \to \neg \chi_0^{(P_1^-, P_1^+)}[\mathsf{V}_1]$. 
Hence, there exists a truth assignment $\mathsf{V}_0$ such that $\mathsf{V}_0(\varphi) = 1$ and $\mathsf{V}_0 \bigl(\chi_0^{(P_1^-, P_1^+)}[\mathsf{V}_1] \bigr) = 1$. 
By Lemma \ref{lem:equiv}, we get $\mathsf{V}_0 \xrightarrow[0]{(P_1^+, P_1^-)} \mathsf{V}_1$. 

It follows from Theorem \ref{basis} that there exists a truth assignment $\mathsf{V}^*$ such that $\mathsf{V}_0 \xrightarrow[0]{(P_0^+ \cup P_1^+, P_0^- \cup P_1^-)} \mathsf{V}^*$ and $\mathsf{V}^* \xrightarrow[0]{(P_1^+ \cup P_2^+, P_1^- \cup P_2^-)} \mathsf{V}_1$. 
Since $\varphi$ is a $(P_0^+ \cup P_1^+, P_0^- \cup P_1^-)$-formula and $\rho$ is a $(P_1^+ \cup P_2^+, P_1^- \cup P_2^-)$-formula, we obtain $\mathsf{V}^*(\varphi) = 1$ and $\mathsf{V}^*(\rho) = 0$. 
Therefore, $\mathsf{V}^*(\varphi \to \rho) = 0$. 
We conclude $\Cl \nvdash \varphi \to \rho$. 
\end{proof}

\section{A sufficient condition for ULIP in modal logic}\label{sec:ULIP}

In this section, we present a sufficient condition for modal logics to have ULIP. 
For this purpose, we extend Definition \ref{defn:basic} to the framework of modal propositional logic. 
From now on, we assume that formulas mean formulas of modal propositional logic.
The language of modal propositional logic is obtained from that of classical propositional logic by adding the unary modal operator $\Box$. 
Let $\Diamond \varphi$ be the abbreviation for $\neg \Box \neg \varphi$. 
The notions of positive variables and negative variables of formulas are extended to the language of modal propositional logic with the following clause: 
\begin{itemize}
\item $v^\circ(\Box \varphi) = v^\circ(\varphi)$ for $\circ \in \{+, -\}$. 
\end{itemize}

\begin{defn}
For each formula $\varphi$, let $d(\varphi)$ denote the maximum number of nested occurrences of $\Box$ in $\varphi$. 
More precisely, $d(\varphi)$ is defined recursively as follows: 
\begin{itemize}
    \item $d(p) = d(\bot) = 0$, where $p$ is a propositional variable $p$, 
    \item $d(\neg \varphi) = d(\varphi)$, 
    \item $d(\varphi \circ \psi) = \max\{d(\varphi), d(\psi)\}$ for $\circ \in \{\land, \lor, \to\}$, 
    \item $d(\Box \varphi) = d(\varphi) + 1$. 
\end{itemize}
\end{defn}

\begin{defn}\leavevmode
    \begin{itemize}
         \item A pair $(W, R)$ is called a ($\SF$-)\emph{Kripke frame} iff $W$ is a non-empty set and $R$ is a reflexive and transitive binary relation on $W$. 

         \item A triple $(W, R, \Vdash)$ is called a \emph{Kripke model} iff $(W, R)$ is a Kripke frame and $\Vdash$ is a satisfaction relation between elements of $W$ and formulas fulfilling the usual conditions for each propositional connective and the following condition:
    \[
        (M, x) \Vdash \Box \varphi \iff \forall y (x R y \Rightarrow (M, y) \Vdash \varphi). 
    \]
\end{itemize}
\end{defn}

\begin{defn}\label{defn:basic2}
    Let $P^+$ and $P^-$ be any finite sets of propositional variables and $n$ be any natural number. 
\begin{enumerate}
    \item A formula $\varphi$ is said to be a \emph{$(P^+, P^-)$-formula} iff $v^\circ(\varphi) \subseteq P^\circ$ for $\circ \in \{+, -\}$. 

    \item Let $\Fm{n}$ denote a fixed finite set of $(P^+, P^-)$-formulas $\varphi$ with $d(\varphi) \leq n$ satisfying that for any $(P^+, P^-)$-formula $\psi$ with $d(\psi) \leq n$, there exists a $\rho \in \Fm{n}$ such that $\K \vdash \psi \leftrightarrow \rho$. 
    The existence of such a finite set $\Fm{n}$ is proved by induction on $n$. 
    
    \item For Kripke models $M_0 = (W_0, R_0, \Vdash_0)$ and $M_1  = (W_1, R_1, \Vdash_1)$ and elements $w_0 \in W_0$ and $w_1 \in W_1$, we write
    \[
        (M_0, w_0) \xr{n} (M_1, w_1)
    \]
    iff for any $(P^+, P^-)$-formula $\varphi$ with $d(\varphi) \leq n$, if $(M_0, w_0) \Vdash_0 \varphi$, then $(M_1, w_1) \Vdash_1 \varphi$. 

    \item For every Kripke model $M = (W, R, \Vdash)$ and element $w \in W$, let $\Cj{n}[M, w]$ denote the $(P^+, P^-)$-formula
    \[
        \bigwedge \{\varphi \in \Fm{n} \mid (M, w) \Vdash \varphi\}.
    \]
    \end{enumerate}
\end{defn}

\begin{lem}[Cf.~{\cite[Proposition 6]{Kurahashi20}}]\label{lem:equiv2}
Let $P^+$ and $P^-$ be any finite sets of propositional variables, $M_0 = (W_0, R_0, \Vdash_0)$ and $M_1  = (W_1, R_1, \Vdash_1)$ be Kripke models, $w_0 \in W_0$ and $w_1 \in W_1$, and $n$ be any natural number. 
Then, the following are equivalent: 
\begin{enumerate}
    \item $(M_0, w_0) \xr{n} (M_1, w_1)$. 
    \item $(M_1, w_1) \xrightarrow[n]{(P^-, P^+)} (M_0, w_0)$. 
    \item $(M_1, w_1) \Vdash_1 \Cj{n}[M_0, w_0]$. 
    \item $(M_0, w_0) \Vdash_0 \Cjc{n}[M_1, w_1]$. 
\end{enumerate}
\end{lem}

\begin{defn}[p-morphisms]
    Let $F = (W, R)$ and $F'  = (W', R')$ be Kripke frames. 
    A mapping $f: W \to W'$ is called a \emph{p-morphism} iff the following conditions hold: 
    \begin{itemize}
        \item For any $x, y \in W$, if $x R y$, then $f(x) R' f(y)$. 
        \item For any $x \in W$ and $w \in W'$, if $f(x) R' w$, then there exists $z \in W$ such that $x R z$ and $f(z) = w$. 
    \end{itemize}
\end{defn}

We introduce a key notion in this paper.

\begin{defn}[$n$-IP]\label{def:IP}
    Let $\mathcal{C}$ be a class of Kripke models and $n$ be a natural number. 
    We say that $\mathcal{C}$ enjoys \emph{$n$-IP} iff for any finite sets $P^+$ and $P^-$ of propositional variables, any Kripke models $M_0 = (W_0, R_0, \Vdash_0)$ and $M_1  = (W_1, R_1, \Vdash_1)$ in $\mathcal{C}$, and any elements $w_0 \in W_0$ and $w_1 \in W_1$, if $(M_0, w_0) \xr{n} (M_1, w_1)$, then there exist a Kripke frame $\mathcal{F}^* = (W^*, R^*)$, an element $w^* \in W^*$, and two $p$-morphisms $f_0: W^* \to W_0$ and $f_1: W^* \to W_1$ satisfying the following three conditions:
    \begin{enumerate}
        \item all Kripke models based on $\mathcal{F}^*$ are in $\mathcal{C}$, 
        \item $f_0(w^*) = w_0$ and $f_1(w^*) = w_1$, 
        \item for any $x^* \in W^*$, we have $(M_0, f_0(x^*)) \xr{0} (M_1, f_1(x^*))$. 
    \end{enumerate}
\end{defn}

We are ready to prove our main instrument in the present paper.
The following theorem is an adaptation of Theorem 2 in \cite{Kurahashi20} for the purpose of this paper, and it also extends the proof of Theorem \ref{Cl_ULIP} to modal logic. 
We note that our sufficient condition for ULIP resembles to Marx's condition on bisimulation products for CIP \cite{Marks99}. 

\begin{thm}\label{thm:main}
If a modal logic $L$ is sound and complete with respect to some class $\mathcal{C}$ of Kripke models enjoying $n$-IP for some natural number $n$, then $L$ has ULIP. 
\end{thm}
\begin{proof}
Suppose that a class $\mathcal{C}$ of Kripke models enjoys $n$-IP and that $L$ is sound and complete with respect to $\mathcal{C}$. 
Let $\varphi$ be any formula and $P^+$ and $P^-$ be any finite sets of propositional variables. 
Let $P_0^\circ : = P^\circ$ and $P_1^\circ : = v^\circ(\varphi) \setminus P^\circ$ for $\circ \in \{+, -\}$. 
Let $\theta$ be the $(P_1^+, P_1^-)$-formula
\[
    \bigwedge \{\psi \in \Phi_n^{(P_1^+, P_1^-)} \mid L \vdash \varphi \to \psi\}.
\]
We prove that $\theta$ is a uniform Lyndon interpolant of $(\varphi, P^+, P^-)$ in $L$. 
The conditions $L \vdash \varphi \to \theta$ and $v^\circ(\theta) \subseteq v^\circ(\varphi) \setminus P^\circ$ for $\circ \in \{+, -\}$ are easily verified. 
Let $\rho$ be any formula such that $L \nvdash \theta \to \rho$ and $v^\circ(\rho) \cap P^\circ = \emptyset$ for $\circ \in \{+, -\}$. 
We prove $L \nvdash \varphi \to \rho$. 
Let $P_2^\circ : = v^\circ(\rho) \setminus v^\circ(\varphi)$ for $\circ \in \{+, -\}$. 
Since $v^\circ(\rho) \cap P^\circ = \emptyset$, we have that $P_0^\circ$, $P_1^\circ$, and $P_2^\circ$ are pairwise disjoint for $\circ \in \{+, -\}$. 

Since $L \nvdash \theta \to \rho$, by the completeness, there exists a Kripke model $M_1 = (W_1, R_1, \Vdash_1)$ in $\mathcal{C}$ and $w_1 \in W_1$ such that $(M_1, w_1) \Vdash_1 \theta$ and $(M_1, w_1) \nVdash_1 \rho$. 
Since $(M_1, w_1) \Vdash_1 \chi_n^{(P_1^-, P_1^+)}[M_1, w_1]$, we have
\[
    (M_1, w_1) \nVdash_1 \theta \to \neg \chi_n^{(P_1^-, P_1^+)}[M_1, w_1].
\]
By the soundness, $L \nvdash \theta \to \neg \chi_n^{(P_1^-, P_1^+)}[M_1, w_1]$. 
Since $\neg \chi_n^{(P_1^-, P_1^+)}[M_1, w_1]$ is a $(P_1^+, P_1^-)$-formula, by the definition of $\theta$, we obtain $L \nvdash \varphi \to \neg \chi_n^{(P_1^-, P_1^+)}[M_1, w_1]$. 
By the completeness, there exists a Kripke model $M_0 = (W_0, R_0, \Vdash_0)$ in $\mathcal{C}$ and $w_0 \in W_0$ such that $(M_0, w_0) \Vdash_0 \varphi$ and $(M_0, w_0) \Vdash_0 \chi_n^{(P_1^-, P_1^+)}[M_1, w_1]$. 
By Lemma \ref{lem:equiv2}, we get $(M_0, w_0) \xrightarrow[n]{(P_1^+, P_1^-)} (M_1, w_1)$. 

Since $\mathcal{C}$ enjoys $n$-IP, there exists a Kripke frame $\mathcal{F}^* = (W^*, R^*)$, an element $w^* \in W^*$, and two $p$-morphisms $f_0: W^* \to W_0$ and $f_1: W^* \to W_1$ satisfying the following three conditions:
    \begin{enumerate}
        \item all Kripke models based on $\mathcal{F}^*$ are in $\mathcal{C}$, 
        \item $f_0(w^*) = w_0$ and $f_1(w^*) = w_1$, 
        \item for any $x^* \in W^*$, we have $(M_0, f_0(x^*)) \xrightarrow[0]{(P_1^+, P_1^-)} (M_1, f_1(x^*))$. 
    \end{enumerate}

We define a Kripke model $M^* = (W^*, R^*, \Vdash^*)$ as follows: \\
Let $x^* \in W^*$.
Since $(M_0, f_0(x^*)) \xrightarrow[0]{(P_1^+, P_1^-)} (M_1, f_1(x^*))$, it follows from Theorem \ref{basis} that there exists a truth assignment $\mathsf{V}^*$ such that for any formula $\psi$ with $d(\psi) = 0$, the following two conditions hold: 
    \begin{enumerate}
        \item If $\psi$ is a $(P_0^+ \cup P_1^+, P_0^- \cup P_1^-)$-formula and $(M_0, f_0(x^*)) \Vdash_0 \psi$, then $\mathsf{V}^*(\psi) = 1$. 
        \item If $\psi$ is a $(P_1^+ \cup P_2^+, P_1^- \cup P_2^-)$-formula and $\mathsf{V}^*(\psi) = 1$, then $(M_1, f_1(x^*)) \Vdash_1 \psi$. 
    \end{enumerate}
    Let $(M^*, x^*) \Vdash^* p: \iff \mathsf{V}^*(p) = 1$. 

    By the first condition of $n$-IP, we have $M^* \in \mathcal{C}$. 

We prove that for any formula $\psi$ and any $x^* \in W^*$, 
\begin{enumerate}
        \item If $\psi$ is a $(P_0^+ \cup P_1^+, P_0^- \cup P_1^-)$-formula and $(M_0, f_0(x^*)) \Vdash_0 \psi$, then $(M^*, x^*) \Vdash^* \psi$. 
        \item If $\psi$ is a $(P_0^- \cup P_1^-, P_0^+ \cup P_1^+)$-formula and $(M_0, f_0(x^*)) \nVdash_0 \psi$, then $(M^*, x^*) \nVdash^* \psi$. 
\end{enumerate}
We simultaneously prove these two statements by induction on the construction of $\psi$.  
The case that $d(\psi) = 0$ directly follows from Theorem \ref{basis}. 
The cases of propositional connectives easily follow from the induction hypothesis. 
So we give only the proof of the case that $\psi$ is of the form $\Box \xi$. 

\medskip

1. Suppose that $\Box \xi$ is a $(P_0^+ \cup P_1^+, P_0^- \cup P_1^-)$-formula and $(M^*, x^*) \nVdash^* \Box \xi$. 
Then, there exists $y^* \in W^*$ such that $x^* R^* y^*$ and $(M^*, y^*) \nVdash^* \xi$. 
Since $\xi$ is also a $(P_0^+ \cup P_1^+, P_0^- \cup P_1^-)$-formula, by the induction hypothesis, we have $(M_0, f_0(y^*)) \nVdash_0 \xi$. 
We have $f_0(x^*) R_0 f_0(y^*)$ because $f_0$ is a p-morphism, and so we obtain $(M_0, f_0(x^*)) \nVdash_0 \Box \xi$. 

\medskip

2. Suppose that $\Box \xi$ is a $(P_0^- \cup P_1^-, P_0^+ \cup P_1^+)$-formula and $(M_0, f_0(x^*)) \nVdash_0 \Box \xi$. 
There exists $y_0 \in W_0$ such that $f_0(x^*) R_0 y_0$ and $(M_0, y_0) \nVdash_0 \xi$. 
Since $f_0$ is a p-morphism, there exists $y^* \in W^*$ such that $x^* R^* y^*$ and $f_0(y^*) = y_0$. 
Since $\xi$ is also a $(P_0^- \cup P_1^-, P_0^+ \cup P_1^+)$-formula, by the induction hypothesis, we have $(M^*, y^*) \nVdash^* \xi$. 
Hence, $(M^*, x^*) \nVdash^* \Box \xi$. 

\medskip

In the same way, we can prove that for any formula $\psi$ and any $x^* \in W^*$, 
\begin{itemize}
        \item If $\psi$ is a $(P_1^+ \cup P_2^+, P_1^- \cup P_2^-)$-formula and $(M^*, x^*) \Vdash^* \psi$, then $(M_1, f_1(x^*)) \Vdash_1 \psi$. 
\end{itemize}
Since $\varphi$ is a $(P_0^+ \cup P_1^+, P_0^- \cup P_1^-)$-formula and $(M_0, w_0) = (M_0, f_0(w^*)) \Vdash_0 \varphi$, we have $(M^*, w^*) \Vdash^* \varphi$. 
Also, since $\rho$ is a $(P_1^+ \cup P_2^+, P_1^- \cup P_2^-)$-formula and $(M_1, w_1) = (M_1, f_1(w^*)) \nVdash_1 \rho$, we obtain $(M^*, w^*) \nVdash^* \rho$. 
Therefore, $(M^*, w^*) \nVdash^* \varphi \to \rho$. 
By the soundness, we conclude $L \nvdash \varphi \to \rho$. 
\end{proof}

\begin{rem}\label{labeling}
In Definition \ref{def:IP}, if we label each $x^* \in W^*$ with the pair $(f_0(x^*), f_1(x^*))$ and we roughly write $x^*$ as its label $(f_0(x^*), f_1(x^*))$, then the conditions required for $\mathcal{F}^*$ in the definition can be rewritten as follows without mentioning p-morphisms:
\begin{enumerate}
    \item all Kripke models based on $\mathcal{F}^*$ are in $\mathcal{C}$, 
    \item $(w_0, w_1) \in W^*$, 
    \item if $(x_0, x_1) R^* (y_0, y_1)$, then $x_0 R_0 y_0$ and $x_1 R_1 y_1$, 
    \item if $(x_0, x_1) \in W^*$ and $x_0 R_0 y_0$, then $(x_0, x_1) R^* (y_0, y_1)$ for some $y_1 \in W_1$, 
    \item if $(x_0, x_1) \in W^*$ and $x_1 R_1 y_1$, then $(x_0, x_1) R^* (y_0, y_1)$ for some $y_0 \in W_0$, 
    \item if $(x_0, x_1) \in W^*$, then $(M_0, x_0) \xr{0} (M_1, x_1)$. 
\end{enumerate}
Since there may be distinct elements of $W^*$ with the same label, the above description is somewhat imprecise. 
For example, Clause 4 should be precisely as follows:
\begin{itemize}
    \item if $x^* \in W^*$ has a label $(x_0, x_1)$ and $x_0 R_0 y_0$, then there exist $y^* \in W^*$ and $y_1 \in W_1$ such that $y^*$ has the label $(y_0, y_1)$ and $x^* R^* y^*$. 
\end{itemize}
Since this rough labeling notation makes it easy to check that a given model meets the required conditions, in the subsequent sections, we will adopt this notation to describe the elements of $W^*$. 
\end{rem}

\section{Some lemmas on matching between clusters}\label{sec:lemmas}

Before actually proving ULIP for several logics using the method developed in the last section, in this section, we present some lemmas on matching between clusters, which will be used in our proofs of ULIP.  

For each Kripke frame $(W, R)$, we say that a subset $I \subseteq W$ is a \emph{cluster} iff $I$ is an equivalence class of the equivalence relation $R^s$ defined by $x R^s y : \iff x R y\ \&\ y R x$. 
All the modal logics for which we prove their ULIP in this paper are extensions of $\Gamma(\LP_2, \omega, 1)$, and so we will basically deal only with finite Kripke models $(W, R, \Vdash)$ such that $W$ may have or may not have the unique root element $x$ and $W \setminus \{x\}$ is a disjoint union of finitely many clusters of final elements. 
For two such models $M_0 = (W_0, R_0, \Vdash_0)$, $M_1 = (W_1, R_1, \Vdash_1)$ and elements $w_0 \in W_0$, $w_1 \in W_1$, we discuss some conclusions that can be obtained from the assumption that $(M_0, w_0) \xr{n} (M_1, w_1)$ holds. 
In this section, for $i \in \{0, 1\}$, let $x_i$ be the root element of $M_i$ and let $C_i^0, C_i^1, \ldots, C_i^{j_i}$ be all the clusters of final elements of $M_i$. 
The model $M_i$ is visualized in Figure \ref{fig:models}: 

\begin{figure}[ht]
\centering
\begin{tikzpicture}
\draw (-3, 2) circle (0.7 and 0.5);
\draw (-1, 2) circle (0.7 and 0.5);
\draw (3, 2) circle (0.7 and 0.5);
\draw (1, 2) node{$\cdots$};
\draw (-3, 2.5) node[above]{$C_i^0$};
\draw (-1, 2.5) node[above]{$C_i^1$};
\draw (3, 2.5) node[above]{$C_i^{j_i}$};

\fill (-3.5, 2) circle (2pt);
\fill (-3, 2) circle (2pt);
\draw (-2.5, 2) node{$\cdots$};

\fill (0, 0) circle (2pt);
\draw (-0.3, 0) node[left]{$x_i$};

\draw[<->] (-3.4,2)--(-3.1, 2);
\draw[->] (0,0)--(-3.5, 1.9);
\draw[->] (0,0)--(-3, 1.9);
\draw[->] (0,0)--(-1.5, 1.9);
\draw[->] (0,0)--(-1, 1.9);
\draw[->] (0,0)--(2.5, 1.9);
\draw[->] (0,0)--(3, 1.9);
\end{tikzpicture}
\caption{The model $M_i = (W_i, R_i, \Vdash_i)$}\label{fig:models}
\end{figure}
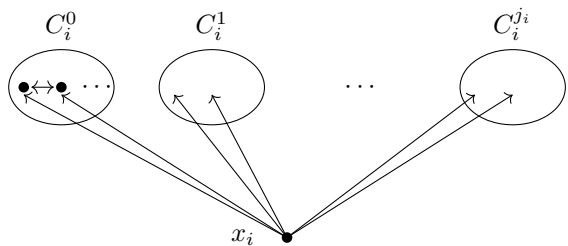

\begin{defn}\label{match}
We say that a cluster $C_0^k$ of $M_0$ \emph{matches} a cluster $C_1^l$ of $M_1$ iff the following two conditions hold:
\begin{enumerate}
    \item For any $u_0 \in C_0^k$, there exists $u_1 \in C_1^l$ such that $(M_0, u_0) \xr{0} (M_1, u_1)$.
    \item For any $u_1 \in C_1^l$, there exists $u_0 \in C_0^k$ such that $(M_0, u_0) \xr{0} (M_1, u_1)$.
\end{enumerate}
\end{defn}

\begin{lem}\label{lem:match_1}
For $w_0 \in C_0^k$ and $w_1 \in C_1^l$, if $(M_0, w_0) \xr{1} (M_1, w_1)$, then $C_0^k$ matches $C_1^l$. 
\end{lem}
\begin{proof}
Suppose $(M_0, w_0) \xr{1} (M_1, w_1)$. 
We prove the two clauses of Definition \ref{match}. 

1. For each $u_0 \in C_0^k$, since $(M_0, u_0) \Vdash_0 \Cj{0}[M_0, u_0]$, we have $(M_0, w_0) \Vdash_0 \Diamond \Cj{0}[M_0, u_0]$ because $w_0 R_0 u_0$. 
Since $(M_0, w_0) \xr{1} (M_1, w_1)$, this formula is also true in $(M_1, w_1)$. 
Thus, we find some $u_1 \in C_1^l$ such that $(M_1, u_1) \Vdash_1 \Cj{0}[M_0, u_0]$. 
By Lemma \ref{lem:equiv2}, we obtain $(M_0, u_0) \xr{0} (M_1, u_1)$.

2. For each $u_1 \in C_1^l$, we have $(M_1, w_1) \Vdash_1 \Diamond \Cjc{0}[M_1, u_1]$. 
Since $(M_0, w_0) \xr{1} (M_1, w_1)$, this formula is also true in $(M_0, w_0)$. 
Then as above, we find some $u_0 \in C_0^k$ such that $(M_0, u_0) \xr{0} (M_1, u_1)$.
\end{proof}

\begin{lem}\label{lem:match_2}
For $w_0 \in C_0^k$, if $(M_0, w_0) \xr{2} (M_1, x_1)$, then $C_0^k$ matches all the clusters $C_1^l$ of $M_1$. 
\end{lem}
\begin{proof}
Suppose $(M_0, w_0) \xr{2} (M_1, x_1)$. 
Let $C_1^l$ be any cluster of $M_1$ and $u_1$ be any element of $C_1^l$. 
Since $(M_1, u_1) \Vdash_1 \Cjc{1}[M_1, u_1]$, we have $(M_1, x_1) \Vdash_1 \Diamond \Cjc{1}[M_1, u_1]$. 
By the supposition, we have that this formula is also true in $(M_0, w_0)$. 
So, there exists some $u_0 \in C_0^k$ such that $(M_0, u_0) \xr{1} (M_1, u_1)$. 
By Lemma \ref{lem:match_1}, we have that $C_0^k$ matches $C_1^l$. 
\end{proof}

\begin{lem}\label{lem:match_3}
If $(M_0, x_0) \xr{3} (M_1, x_1)$, then the following two properties hold: 
\begin{enumerate}
    \item For any cluster $C_0^k$ of $M_0$, there exists a cluster $C_1^l$ of $M_1$ such that $C_0^k$ matches $C_1^l$. 
    \item For any cluster $C_1^l$ of $M_1$, there exists a cluster $C_0^k$ of $M_0$ such that $C_0^k$ matches $C_1^l$. 
\end{enumerate}
\end{lem}
\begin{proof}
Suppose $(M_0, x_0) \xr{3} (M_1, x_1)$. 
We give only a proof of the first clause and the second clause can be proved in the same way. 

1. Let $C_0^k$ be any cluster of $M_0$. 
For any $v_0 \in C_0^k$, we have 
\[
    (M_0, v_0) \Vdash_0 \bigvee_{u_0 \in C_0^k} \Cj{1}[M_0, u_0]
\]
and hence 
\[
    (M_0, v_0) \Vdash_0 \Box \bigvee_{u_0 \in C_0^k} \Cj{1}[M_0, u_0].
\]
Then, 
\[
    (M_0, x_0) \Vdash_0 \Diamond \Box \bigvee_{u_0 \in C_0^k} \Cj{1}[M_0, u_0].
\]
Since $(M_0, x_0) \xr{3} (M_1, x_1)$, this formula is also true in $(M_1, x_1)$. 
Then, for some $y_1 \in W_1$, 
\[
    (M_1, y_1) \Vdash_1 \Box \bigvee_{u_0 \in C_0^k} \Cj{1}[M_0, u_0].
\]
Hence, we find some cluster $C_1^l$ of $M_1$ and an element $u_1 \in C_1^l$ such that
\[
    (M_1, u_1) \Vdash_1 \bigvee_{u_0 \in C_0^k} \Cj{1}[M_0, u_0].
\]
Thus, there exists $u_0 \in C_0^k$ such that $(M_0, u_0) \xr{1} (M_1, u_1)$. 
By Lemma \ref{lem:match_1}, we conclude that $C_0^k$ matches $C_1^l$. 
\end{proof}

\begin{lem}\label{lem:mariage}
Let $A = \{a_0, a_1\}$ and $B = \{b_0, b_1\}$ be any sets consisting of two elements and let $R \subseteq A \times B$ be any binary relation. 
Suppose that $R$ satisfies the following two conditions: 
\begin{enumerate}
    \item For any $a \in A$, there exists $b \in B$ such that $(a, b) \in R$. 
    \item For any $b \in B$, there exists $a \in A$ such that $(a, b) \in R$. 
\end{enumerate}
Then, there exist $c, d \in B$ such that $c \neq d$, $(a_0, c) \in R$, and $(a_1, d) \in R$. 
\end{lem}
\begin{proof}
If both $(a_0, b_0)$ and $(a_1, b_1)$ are in $R$, then the statement holds for $c = b_0$ and $d = b_1$. 
If $(a_0, b_0) \notin R$ or $(a_1, b_1) \notin R$, then by considering the two conditions, we find that both $(a_0, b_1)$ and $(a_1, b_0)$ are in $R$, and so the statement holds for $c = b_1$ and $d = b_0$. 
\end{proof}

\begin{lem}\label{lem:match_two}
If the clusters $C_0^k = \{y_0, z_0\}$ and $C_1^l = \{y_1, z_1\}$ consist of two elements and $C_0^k$ matches $C_1^l$, then there exist $u_1, v_1 \in C_1^l$ such that $u_1 \neq v_1$ and
\[
    (M_0, y_0) \xr{0} (M_1, u_1)\ \text{and}\ (M_0, z_0) \xr{0} (M_1, v_1). 
\]
\end{lem}
\begin{proof}
This directly follows from Definition \ref{match} and Lemma \ref{lem:mariage}. 
\end{proof}

\begin{lem}\label{lem:match_two_cl}
Suppose both $M_0$ and $M_1$ include exactly two clusters $C_0^0, C_0^1$ and $C_1^0, C_1^1$ of final elements, respectively. 
If $(M_0, x_0) \xr{3} (M_1, x_1)$, then there exist distinct clusters $D_1^0$ and $D_1^1$ of $M_1$ such that $C_0^k$ matches $D_1^k$ for $k \in \{0, 1\}$. 
\end{lem}
\begin{proof}
This directly follows from Lemmas \ref{lem:match_3} and \ref{lem:mariage}. 
\end{proof}

\section{Modal companions of $\Cl$}\label{sec:Cl}

In this section, we prove ULIP for the logics $\mathbf{Triv}$ and $\mathbf{S5}$ as test cases for our application of Theorem \ref{thm:main}. 

\subsection{$\Triv$}

Let $\mathcal{C}_{\Triv}$ be the class of all Kripke models whose frame consists of a single reflexive world. 
The logic $\Triv$ is sound and complete with respect to $\mathcal{C}_{\Triv}$.

\begin{thm}[Maksimova \cite{Maksimova82}]
$\Triv$ has ULIP. 
\end{thm}
\begin{proof}
We prove that $\mathcal{C}_{\Triv}$ enjoys $0$-IP. 
Then, ULIP for $\Triv$ follows from Theorem \ref{thm:main}. 
Let $P^+$ and $P^-$ be any finite sets of propositional variables and let $M_0 = (\{w_0\}, R_0, \Vdash_0)$ and $M_1  = (\{w_1\}, R_1, \Vdash_1)$ be Kripke models in $\mathcal{C}_{\Triv}$, and suppose $(M_0, w_0) \xr{0} (M_1, w_1)$. 
Let $W^* : = \{(w_0, w_1)\}$ and $R^* : = (W^*)^2$. 
Then, it is easy to check that $\mathcal{F}^* = (W^*, R^*)$ satisfies the conditions stated in Remark \ref{labeling}. 
\end{proof}

\subsection{$\SFi$}

Let $\mathcal{C}_{\SFi}$ be the class of all Kripke models whose frame forms a finite cluster. 
Figure \ref{fig:S5} visualizes such a frame. 
In the following, we assume that each diagram represents the Kripke frame of the transitive and reflexive closure of the arrows shown in the diagram.
The logic $\SFi$ is sound and complete with respect to $\mathcal{C}_{\SFi}$. 

\begin{figure}[ht]
\centering
\begin{tikzpicture}
\draw (1.5, -1.5) circle (2 and 0.5);

\fill (0, -1.5) circle (2pt);
\fill (1, -1.5) circle (2pt);
\draw (2, -1.5) node{$\cdots$};
\fill (3, -1.5) circle (2pt);

\draw[<->, >=Stealth] (0.1,-1.5)--(0.9, -1.5);
\draw[<->, >=Stealth] (1.1,-1.5)--(1.6, -1.5);
\draw[<->, >=Stealth] (2.4,-1.5)--(2.9, -1.5);
\end{tikzpicture}
\caption{A frame of $\SFi$}\label{fig:S5}
\end{figure}

\begin{thm}[Fitting \cite{Fitting83}]
$\SFi$ has ULIP. 
\end{thm}
\begin{proof}
We prove that $\mathcal{C}_{\SFi}$ enjoys $1$-IP. 
Let $P^+$ and $P^-$ be any finite sets of propositional variables and let $M_0 = (W_0, R_0, \Vdash_0)$ and $M_1  = (W_1, R_1, \Vdash_1)$ be Kripke models in $\mathcal{C}_{\SFi}$, let $w_0 \in W_0$ and $w_1 \in W_1$, and suppose $(M_0, w_0) \xr{1} (M_1, w_1)$. 
Then, both $W_0$ and $W_1$ are clusters of $M_0$ and $M_1$, respectively. 
By Lemma \ref{lem:match_1}, we have that $W_0$ matches $W_1$. 

Let $W^* := \{(x_0, x_1) \mid (M_0, x_0) \xr{0} (M_1, x_1)\}$ and $R^* : = (W^*)^2$. 
We check that $\mathcal{F}^* = (W^*, R^*)$ fulfills the conditions stated in Remark \ref{labeling}. 

\begin{enumerate}
    \item Since $\mathcal{F}^*$ forms a finite cluster, all Kripke models based on $\mathcal{F}^*$ are in $\mathcal{C}_{\SFi}$. 
    \item $(w_0, w_1) \in W^*$ because $(M_0, w_0) \xr{0} (M_1, w_1)$. 
    \item If $(x_0, x_1) R^* (y_0, y_1)$, then $x_0, y_0 \in W_0$ and $x_1, y_1 \in W_1$. 
    Since both $W_0$ and $W_1$ form clusters, we obtain $x_0 R_0 y_0$ and $x_1 R_1 y_1$. 
    \item Suppose $(x_0, x_1) \in W^*$ and $x_0 R_0 y_0$. 
    Since $W_0$ matches $W_1$, we find some $y_1 \in W_1$ such that $(M_0, y_0) \xr{0} (M_1, y_1)$. 
    Hence $(y_0, y_1) \in W^*$. 
    Since $R^* = (W^*)^2$, we obtain $(x_0, x_1) R^* (y_0, y_1)$. 
    
    \item Suppose $(x_0, x_1) \in W^*$ and $x_1 R_1 y_1$. 
    In the same way as above, we can find some $y_0 \in W_0$ such that $(y_0, y_1) \in W^*$.
    Then, $(x_0, x_1) R^* (y_0, y_1)$. 
    
    \item If $(x_0, x_1) \in W^*$, then $(M_0, x_0) \xr{0} (M_1, x_1)$ by the definition of $W^*$. 
\end{enumerate}

It follows from Theorem \ref{thm:main} that $\SFi$ has ULIP. 
\end{proof}

\section{Modal companions of $\LS$}\label{sec:LS}

In this section, we prove that the three logics $\mathbf{GW.2} = \Gamma(\LS, 1, 1)$, $\Gamma(\LS, 2, 1)$, and $\mathbf{S4.4} = \Gamma(\LS, \omega, 1)$ have ULIP. 
In particular, LIP for $\Gamma(\LS, 2, 1)$ is new.

\subsection{$\mathbf{GW.2}$}

Let $\mathcal{C}_{\mathbf{GW.2}}$ be the class of all Kripke models whose frames are of the form shown in Figure \ref{fig:GW.2}. 
The logic $\mathbf{GW.2}$ is sound and complete with respect to $\mathcal{C}_{\mathbf{GW.2}}$. 

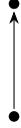
\begin{figure}[ht]
\centering
\begin{tikzpicture}
\fill (0, 0) circle (2pt);
\fill (0, 1.5) circle (2pt);

\draw[->, >=Stealth] (0,0.1)--(0, 1.4);
\end{tikzpicture}
\caption{The frame of $\mathbf{GW.2}$}\label{fig:GW.2}
\end{figure}

\begin{thm}[Maksimova \cite{Maksimova82}]\label{thm:GW.2}
$\mathbf{GW.2}$ has ULIP. 
\end{thm}
\begin{proof}
We prove that $\mathcal{C}_{\mathbf{GW.2}}$ enjoys $2$-IP. 
Let $P^+$ and $P^-$ be any finite sets of propositional variables and let $M_0 = (W_0, R_0, \Vdash_0)$ and $M_1  = (W_1, R_1, \Vdash_1)$ be Kripke models in $\mathcal{C}_{\mathbf{GW.2}}$, let $w_0 \in W_0$ and $w_1 \in W_1$, and suppose $(M_0, w_0) \xr{2} (M_1, w_1)$. 
For $i \in \{0,1\}$, let $y_i$ be the final element of $M_i$. 

We show that $(M_0, y_0) \xr{0} (M_1, y_1)$. 
Since $v_1 R_1 y_1$ for all $v_1 \in W_1$, we have $(M_1, w_1) \Vdash_1 \Box \Diamond \Cjc{0}[M_1, y_1]$. 
Since $(M_0, w_0) \xr{2} (M_1, w_1)$, we get $(M_0, w_0) \Vdash_0 \Box \Diamond \Cjc{0}[M_1, y_1]$. 
Since $w_0 R_0 y_0$, we have $(M_0, y_0) \Vdash_0 \Diamond \Cjc{0}[M_1, y_1]$, and hence $(M_0, y_0) \Vdash_0 \Cjc{0}[M_1, y_1]$. 
By Lemma \ref{lem:equiv2}, $(M_0, y_0) \xr{0} (M_1, y_1)$. 

It is easy to check that the frame $\mathcal{F}^* = (W^*, R^*)$ drawn in Figure \ref{fig:GW.2_cases} satisfies all the required conditions stated in Remark \ref{labeling}. 
Note that in our rough labeling notation, the situation $(w_0, w_1) = (y_0, y_1)$ is allowed. 

\begin{figure}[ht]
\centering
\begin{tikzpicture}
\fill (0, 0) circle (2pt);
\fill (0, 1.5) circle (2pt);
\draw (0, 1.5) node[left]{$(y_0, y_1)$};
\draw (0, 0) node[left]{$(w_0, w_1)$};

\draw[->, >=Stealth] (0,0.1)--(0, 1.4);

\end{tikzpicture}
\caption{The frame $(W^*, R^*)$}\label{fig:GW.2_cases}
\end{figure}
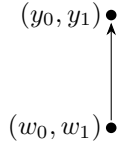
\end{proof}

\subsection{$\Gamma(\LS, 2, 1)$}

Let $\mathcal{C}_{\Gamma(\LS, 2, 1)}$ be the class of all Kripke models whose frames are of the form shown in Figure \ref{fig:LC_2,2,1}. 
The logic $\Gamma(\LS, 2, 1)$ is sound and complete with respect to $\mathcal{C}_{\Gamma(\LS, 2, 1)}$. 

\begin{figure}[ht]
\centering
\begin{tikzpicture}
\fill (0.7, 1.5) circle (2pt);
\fill (-0.7, 1.5) circle (2pt);
\fill (0, 0) circle (2pt);

\draw[->, >=Stealth] (0,0.1)--(-0.7, 1.4);
\draw[->, >=Stealth] (0,0.1)--(0.7, 1.4);
\draw[<->, >=Stealth] (-0.6,1.5)--(0.6, 1.5);
\end{tikzpicture}
\caption{The frame of $\Gamma(\LS, 2, 1)$}\label{fig:LC_2,2,1}
\end{figure}
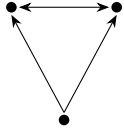

\begin{thm}\label{thm:LS,2,1}
$\Gamma(\LS, 2, 1)$ has ULIP. 
\end{thm}
\begin{proof}
We prove that $\mathcal{C}_{\Gamma(\LS, 2, 1)}$ enjoys $2$-IP. 
Let $P^+$ and $P^-$ be any finite sets of propositional variables and let $M_0 = (W_0, R_0, \Vdash_0)$ and $M_1  = (W_1, R_1, \Vdash_1)$ be Kripke models in $\mathcal{C}_{\Gamma(\LS, 2, 1)}$, let $w_0 \in W_0$ and $w_1 \in W_1$, and suppose $(M_0, w_0) \xr{2} (M_1, w_1)$. 
Let $C_i = \{y_i, z_i\}$ be the cluster of final elements of $M_i$ for $i \in \{0, 1\}$. 

We show that $C_0$ matches $C_1$. 
Let $u_1 \in C_1$. 
    Since $v_1 R_1 u_1$ for all $v_1 \in W_1$ and $w_0 R_0 y_0$, in the same way as in the proof of Theorem \ref{thm:GW.2}, we have $(M_0, y_0) \Vdash_0 \Diamond \Cjc{0}[M_1, u_1]$. 
    Then, $\Cjc{0}[M_1, u_1]$ is true in at least one of $y_0$ and $z_0$.  
    Hence, 
\begin{equation}\label{LC_2,2,1_case1}
    (M_0, y_0) \xr{0} (M_1, u_1)\ \text{or}\ (M_0, z_0) \xr{0} (M_1, u_1). 
\end{equation}
In the similar way, we can prove that for $u_0 \in \{y_0, z_0\}$, 
\begin{equation}\label{LC_2,2,1_case2}
    (M_0, u_0) \xr{0} (M_1, y_1)\ \text{or}\ (M_0, u_0) \xr{0} (M_1, z_1). 
\end{equation}
Since both $C_0$ and $C_1$ consist of two elements, by Lemma \ref{lem:match_two}, there exist $u_1, v_1 \in C_1$ such that $(M_0, y_0) \xr{0} (M_1, u_1)$ and $(M_0, z_0) \xr{0} (M_1, v_1)$

\begin{figure}[ht]
\centering
\begin{tikzpicture}
\fill (0.7, 1.5) circle (2pt);
\fill (-0.7, 1.5) circle (2pt);
\fill (0, 0) circle (2pt);
\draw (-0.7, 1.5) node[above]{$(y_0, u_1)$};
\draw (0.7, 1.5) node[above]{$(z_0, v_1)$};
\draw (0, 0) node[left]{$(w_0, w_1)$};

\draw[->, >=Stealth] (0,0.1)--(-0.7, 1.4);
\draw[->, >=Stealth] (0,0.1)--(0.7, 1.4);
\draw[<->, >=Stealth] (-0.6,1.5)--(0.6, 1.5);

\end{tikzpicture}
\caption{The frame $(W^*, R^*)$}\label{fig:LC_2,2,1_cases}
\end{figure}
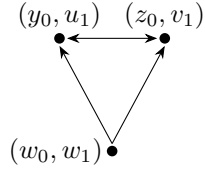

It is easy to check that the frame $(W^*, R^*)$ in Figure \ref{fig:LC_2,2,1_cases} satisfies all the required conditions.
\end{proof}

\subsection{$\mathbf{S4.4}$}

Let $\mathcal{C}_{\mathbf{S4.4}}$ be the class of all finite Kripke models whose frames are of the form shown in Figure \ref{fig:S4.4}. 
The logic $\mathbf{S4.4}$ is sound and complete with respect to $\mathcal{C}_{\mathbf{S4.4}}$.

\begin{figure}[ht]
\centering
\begin{tikzpicture}
\draw (1.5, 1.5) circle (2 and 0.5);

\fill (0, 1.5) circle (2pt);
\fill (1, 1.5) circle (2pt);
\draw (2, 1.5) node{$\cdots$};
\fill (3, 1.5) circle (2pt);

\draw[<->, >=Stealth] (0.1,1.5)--(0.9, 1.5);
\draw[<->, >=Stealth] (1.1,1.5)--(1.6, 1.5);
\draw[<->, >=Stealth] (2.4,1.5)--(2.9, 1.5);

\fill (1.5, 0) circle (2pt);

\draw[->, >=Stealth] (1.5,0.1)--(0, 1.4);
\draw[->, >=Stealth] (1.5,0.1)--(1, 1.4);
\draw[->, >=Stealth] (1.5,0.1)--(3, 1.4);
\end{tikzpicture}
\caption{A frame of $\mathbf{S4.4}$}\label{fig:S4.4}
\end{figure}
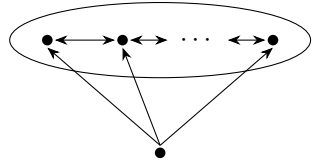

\begin{thm}[Shimura \cite{Shimura92}]\label{thm:S4.4}
$\mathbf{S4.4}$ has ULIP. 
\end{thm}
\begin{proof}
We prove that $\mathcal{C}_{\mathbf{S4.4}}$ enjoys $2$-IP. 
Let $P^+$ and $P^-$ be any finite sets of propositional variables and let $M_0 = (W_0, R_0, \Vdash_0)$ and $M_1  = (W_1, R_1, \Vdash_1)$ be Kripke models in $\mathcal{C}_{\mathbf{S4.4}}$, let $w_0 \in W_0$ and $w_1 \in W_1$, and suppose $(M_0, w_0) \xr{2} (M_1, w_1)$. 
Let $C_i$ be the finite cluster of final elements of $M_i$ for $i \in \{0, 1\}$. 

We show that $C_0$ matches $C_1$. 
For each $v_0 \in C_0$, we have 
\[
    (M_0, v_0) \Vdash_0 \bigvee_{u_0 \in C_0} \Cj{0}[M_0, u_0], 
\]
and hence 
\[
    (M_0, v_0) \Vdash_0 \Box \bigvee_{u_0 \in C_0} \Cj{0}[M_0, u_0].
\]
So, 
\[
    (M_0, w_0) \Vdash_0 \Diamond \Box \bigvee_{u_0 \in C_0} \Cj{0}[M_0, u_0].
\]
Since $(M_0, w_0) \xr{2} (M_1, w_1)$, we obtain that this formula is also true in $(M_1, w_1)$. 
It follows that for every $u_1 \in C_1$, we get
\[
    (M_1, u_1) \Vdash_1 \bigvee_{u_0 \in C_0} \Cj{0}[M_0, u_0], 
\]
and thus, there exists $u_0 \in C_0$ such that $(M_0, u_0) \xr{0} (M_1, u_1)$. 

In the same way, we can show that for any $u_0 \in C_0$, there exists $u_1 \in C_1$ such that $(M_0, u_0) \xr{0} (M_1, u_1)$. 

Let 
\[
    C^* : = \{(u_0, u_1) \mid u_0 \in C_0, u_1 \in C_1,\ \text{and}\ (M_0, u_0) \xr{0} (M_1, u_1)\}.
\]
Then, the frame $(W^*, R^*)$ shown in Figure \ref{fig:S4.4_cases} satisfies all the required conditions.

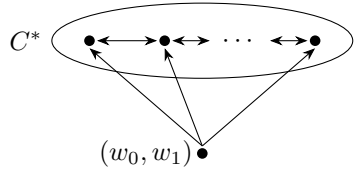
\begin{figure}[ht]
\centering
\begin{tikzpicture}
\draw (1.5, 1.5) circle (2 and 0.5);

\fill (0, 1.5) circle (2pt);
\fill (1, 1.5) circle (2pt);
\draw (2, 1.5) node{$\cdots$};
\fill (3, 1.5) circle (2pt);
\draw (-0.5, 1.5) node[left]{$C^*$};

\draw[<->, >=Stealth] (0.1,1.5)--(0.9, 1.5);
\draw[<->, >=Stealth] (1.1,1.5)--(1.6, 1.5);
\draw[<->, >=Stealth] (2.4,1.5)--(2.9, 1.5);

\fill (1.5, 0) circle (2pt);
\draw (1.5, 0) node[left]{$(w_0, w_1)$};

\draw[->, >=Stealth] (1.5,0.1)--(0, 1.4);
\draw[->, >=Stealth] (1.5,0.1)--(1, 1.4);
\draw[->, >=Stealth] (1.5,0.1)--(3, 1.4);
\end{tikzpicture}
\caption{The frame $(W^*, R^*)$}\label{fig:S4.4_cases}
\end{figure}
\end{proof}

\section{Modal companions of $\LV$}\label{sec:LV}

In this section, we newly prove that the three logics $\GV = \Gamma(\LV, 1, 1)$, $\Gamma(\LV, 2, 1)$, and $\Gamma(\LV, \omega, 1)$ have ULIP. 
In particular, LIP for $\GV$ is mentioned in \cite{Maksimova14} as an open problem. 

\subsection{$\mathbf{GV}$}

Let $\mathcal{C}_{\GV}$ be the class of all Kripke models whose frames are of the form shown in Figure \ref{fig:GV}. 
The logic $\GV$ is sound and complete with respect to $\mathcal{C}_{\GV}$. 

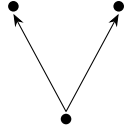
\begin{figure}[ht]
\centering
\begin{tikzpicture}
\fill (0.7, 1.5) circle (2pt);
\fill (-0.7, 1.5) circle (2pt);
\fill (0, 0) circle (2pt);

\draw[->, >=Stealth] (0,0.1)--(-0.7, 1.4);
\draw[->, >=Stealth] (0,0.1)--(0.7, 1.4);
\end{tikzpicture}
\caption{The frame of $\GV$}\label{fig:GV}
\end{figure}

\begin{thm}\label{thm:GV}
$\GV$ has ULIP. 
\end{thm}
\begin{proof}
We prove that $\mathcal{C}_{\GV}$ enjoys $3$-IP. 
Let $P^+$ and $P^-$ be any finite sets of propositional variables and let $M_0 = (W_0, R_0, \Vdash_0)$ and $M_1  = (W_1, R_1, \Vdash_1)$ be Kripke models in $\mathcal{C}_{\GV}$, let $w_0 \in W_0$ and $w_1 \in W_1$, and suppose $(M_0, w_0) \xr{3} (M_1, w_1)$. 
For $i \in \{0, 1\}$, let $x_i$ be the root element of $M_i$ and $y_i, z_i$ be two distinct final elements of $M_i$. 
We distinguish the following three cases. 

\paragraph*{Case 1:} $w_0 \in \{y_0, z_0\}$ and $w_1 \in \{y_1, z_1\}$. 

\paragraph*{Case 2:} $w_0 \in \{y_0, z_0\}$ and $w_1 = x_1$. 

\medskip

For the case where $w_0 = x_0$ and $w_1 \in \{y_1, z_1\}$, we can draw the required frame $(W^*, R^*)$ similar to that corresponding to Case 2 in Figure \ref{fig:GV_cases}. 
In subsequent proofs, we will basically omit the mention of such similar cases.
In Case 2, by Lemma \ref{lem:match_2}, we have that $(M_0, w_0) \xr{0} (M_1, u_1)$ for $u_1 \in \{y_1, z_1\}$. 

\paragraph*{Case 3:} $w_0 = x_0$ and $w_1 = x_1$. 

\medskip

Since both $M_0$ and $M_1$ consist of two clusters, by Lemma \ref{lem:match_two_cl}, there exist two distinct final elements $u_1$ and $v_1$ of $M_1$ such that $(M_0, y_0) \xr{0} (M_1, u_1)$ and $(M_0, z_0) \xr{0} (M_1, v_1)$. 

\medskip

Our frames $(W^*, R^*)$ corresponding to each of the three cases are shown together in Figure \ref{fig:GV_cases}. 
It is easy to check that these frames satisfy all the required conditions.

\begin{figure}[ht]
\centering
\begin{tikzpicture}
\fill (0.7, 1.5) circle (2pt);
\fill (-0.7, 1.5) circle (2pt);
\fill (0, 0) circle (2pt);
\draw (-0.7, 1.5) node[above]{$(w_0, w_1)$};
\draw (0.7, 1.5) node[above]{$(w_0, w_1)$};
\draw (0, 0) node[left]{$(w_0, w_1)$};
\draw (0, -0.5) node[below]{Case 1};

\draw[->, >=Stealth] (0,0.1)--(-0.7, 1.4);
\draw[->, >=Stealth] (0,0.1)--(0.7, 1.4);

\draw [dashed] (2, -1)--(2, 2);

\fill (4.7, 1.5) circle (2pt);
\fill (3.3, 1.5) circle (2pt);
\fill (4, 0) circle (2pt);
\draw (3.3, 1.5) node[above]{$(w_0, y_1)$};
\draw (4.7, 1.5) node[above]{$(w_0, z_1)$};
\draw (4, 0) node[left]{$(w_0, w_1)$};

\draw (4, -0.5) node[below]{Case 2};

\draw[->, >=Stealth] (4,0.1)--(3.3, 1.4);
\draw[->, >=Stealth] (4,0.1)--(4.7, 1.4);

\draw [dashed] (6, -1)--(6, 2);

\fill (8.7, 1.5) circle (2pt);
\fill (7.3, 1.5) circle (2pt);
\fill (8, 0) circle (2pt);
\draw (7.3, 1.5) node[above]{$(y_0, u_1)$};
\draw (8.7, 1.5) node[above]{$(z_0, v_1)$};
\draw (8, 0) node[left]{$(w_0, w_1)$};
\draw (8, -0.5) node[below]{Case 3};

\draw[->, >=Stealth] (8,0.1)--(7.3, 1.4);
\draw[->, >=Stealth] (8,0.1)--(8.7, 1.4);
\end{tikzpicture}
\caption{The frames $(W^*, R^*)$}\label{fig:GV_cases}
\end{figure}
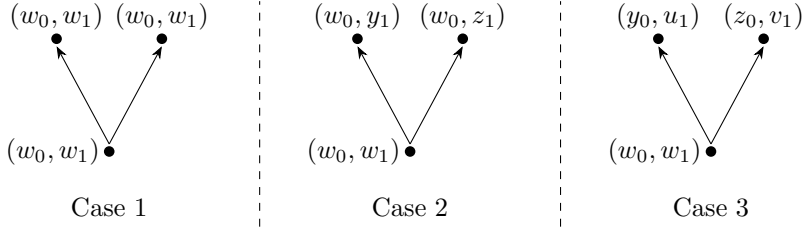
\end{proof}

\subsection{$\Gamma(\LV, 2, 1)$}

Let $\mathcal{C}_{\Gamma(\LV, 2, 1)}$ be the class of all Kripke models whose frames are of the form shown in Figure \ref{fig:LV,2,1}. 
The logic $\Gamma(\LV, 2, 1)$ is sound and complete with respect to $\mathcal{C}_{\Gamma(\LV, 2, 1)}$. 

\begin{figure}[ht]
\centering
\begin{tikzpicture}
\fill (-0.2, 1.5) circle (2pt);
\fill (0.8, 1.5) circle (2pt);
\fill (2.2, 1.5) circle (2pt);
\fill (3.2, 1.5) circle (2pt);

\draw[<->, >=Stealth] (-0.1,1.5)--(0.7, 1.5);
\draw[<->, >=Stealth] (2.3,1.5)--(3.1, 1.5);

\fill (1.5, 0) circle (2pt);

\draw[->, >=Stealth] (1.5,0.1)--(-0.2, 1.4);
\draw[->, >=Stealth] (1.5,0.1)--(0.8, 1.4);
\draw[->, >=Stealth] (1.5,0.1)--(2.2, 1.4);
\draw[->, >=Stealth] (1.5,0.1)--(3.2, 1.4);
\end{tikzpicture}
\caption{The frame of $\Gamma(\LV, 2, 1)$}\label{fig:LV,2,1}
\end{figure}
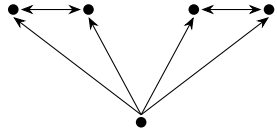

\begin{thm}\label{thm:LV,2,1}
$\Gamma(\LV, 2, 1)$ has ULIP. 
\end{thm}
\begin{proof}
We prove that $\mathcal{C}_{\Gamma(\LV, 2, 1)}$ enjoys $3$-IP. 
Let $P^+$ and $P^-$ be any finite sets of propositional variables and let $M_0 = (W_0, R_0, \Vdash_0)$ and $M_1  = (W_1, R_1, \Vdash_1)$ be Kripke models in $\mathcal{C}_{\Gamma(\LV, 2, 1)}$, let $w_0 \in W_0$ and $w_1 \in W_1$, and suppose $(M_0, w_0) \xr{3} (M_1, w_1)$. 
For $i \in \{0, 1\}$, let $x_i$ be the root element of $M_i$ and let $C_i^0 = \{y_i^0, z_i^0\}$ and $C_i^1 = \{y_i^1, z_i^1\}$ be two distinct clusters of final elements of $M_i$. 
We distinguish the following three cases. 

\paragraph*{Case 1:} $w_0 \in C_0^0$ and $w_1 \in C_1^0$. 

\medskip

By Lemma \ref{lem:match_1}, we have that $C_0^0$ matches $C_1^0$. 
Since each cluster consists of two elements, by Lemma \ref{lem:match_two}, there exist $u_1^0, v_1^0 \in C_1^0$ such that $(M_0, y_0^0) \xr{0} (M_1, u_1^0)$ and $(M_0, z_0^0) \xr{0} (M_1, v_1^0)$. 

\paragraph*{Case 2:} $w_0 \in C_0^0$ and $w_1 = x_1$. 

\medskip

By Lemma \ref{lem:match_2}, $C_0^0$ matches both $C_1^0$ and $C_1^1$. 
Then, by Lemma \ref{lem:match_two}, for $k \in \{0, 1\}$, there exist $u_1^k, v_1^k \in C_1^k$ such that $(M_0, y_0^0) \xr{0} (M_1, u_1^k)$ and $(M_0, z_0^0) \xr{0} (M_1, v_1^k)$.

\paragraph*{Case 3:} $w_0 = x_0$ and $w_1 = x_1$. 

\medskip

Since both $M_0$ and $M_1$ consist of two clusters, by Lemma \ref{lem:match_two_cl}, there exist two distinct clusters $D_1^0$ and $D_1^1$ of final elements of $M_1$ such that for $k \in \{0, 1\}$, $C_0^k$ matches $D_1^k$. 
Also, since each cluster consist of two elements, by Lemma \ref{lem:match_two}, for $k \in \{0, 1\}$, there exist $u_1^k, v_1^k \in D_1^k$ such that $(M_0, y_0^k) \xr{0} (M_1, u_1^k)$ and $(M_0, z_0^k) \xr{0} (M_1, v_1^k)$. 

\medskip

Our frames $(W^*, R^*)$ corresponding to Cases 1 and 2 are shown together in Figure \ref{fig:LV,2,1_case12}. 
Figure \ref{fig:LV,2,1_case3} shows our frame $(W^*, R^*)$ corresponding to Case 3. 
These frames satisfy all the required conditions.

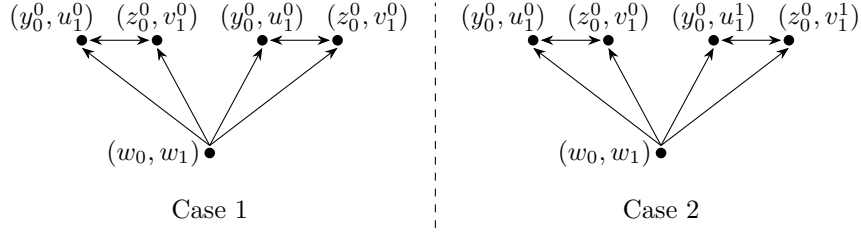
\begin{figure}[ht]
\centering
\begin{tikzpicture}

\fill (-0.2, 1.5) circle (2pt);
\fill (0.8, 1.5) circle (2pt);
\fill (2.2, 1.5) circle (2pt);
\fill (3.2, 1.5) circle (2pt);
\draw (-0.6, 1.5) node[above]{$(y_0^0, u_1^0)$};
\draw (0.8, 1.5) node[above]{$(z_0^0, v_1^0)$};
\draw (2.2, 1.5) node[above]{$(y_0^0, u_1^0)$};
\draw (3.6, 1.5) node[above]{$(z_0^0, v_1^0)$};
\draw (1.5, -0.5) node[below]{Case 1};

\draw[<->, >=Stealth] (-0.1,1.5)--(0.7, 1.5);
\draw[<->, >=Stealth] (2.3,1.5)--(3.1, 1.5);

\fill (1.5, 0) circle (2pt);
\draw (1.5, 0) node[left]{$(w_0, w_1)$};

\draw[->, >=Stealth] (1.5,0.1)--(-0.2, 1.4);
\draw[->, >=Stealth] (1.5,0.1)--(0.8, 1.4);
\draw[->, >=Stealth] (1.5,0.1)--(2.2, 1.4);
\draw[->, >=Stealth] (1.5,0.1)--(3.2, 1.4);

\draw [dashed] (4.5, -1)--(4.5, 2);

\fill (5.8, 1.5) circle (2pt);
\fill (6.8, 1.5) circle (2pt);
\fill (8.2, 1.5) circle (2pt);
\fill (9.2, 1.5) circle (2pt);
\draw (5.4, 1.5) node[above]{$(y_0^0, u_1^0)$};
\draw (6.8, 1.5) node[above]{$(z_0^0, v_1^0)$};
\draw (8.2, 1.5) node[above]{$(y_0^0, u_1^1)$};
\draw (9.6, 1.5) node[above]{$(z_0^0, v_1^1)$};

\draw (7.5, -0.5) node[below]{Case 2};

\draw[<->, >=Stealth] (5.9,1.5)--(6.7, 1.5);
\draw[<->, >=Stealth] (8.3,1.5)--(9.1, 1.5);

\fill (7.5, 0) circle (2pt);
\draw (7.5, 0) node[left]{$(w_0, w_1)$};

\draw[->, >=Stealth] (7.5,0.1)--(5.8, 1.4);
\draw[->, >=Stealth] (7.5,0.1)--(6.8, 1.4);
\draw[->, >=Stealth] (7.5,0.1)--(8.2, 1.4);
\draw[->, >=Stealth] (7.5,0.1)--(9.2, 1.4);
\end{tikzpicture}
\caption{The frames $(W^*, R^*)$ in Cases 1 and 2}\label{fig:LV,2,1_case12}
\end{figure}

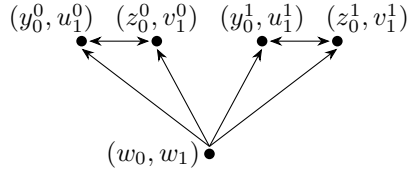
\begin{figure}[ht]
\centering
\begin{tikzpicture}
\fill (-0.2, 1.5) circle (2pt);
\fill (0.8, 1.5) circle (2pt);
\fill (2.2, 1.5) circle (2pt);
\fill (3.2, 1.5) circle (2pt);
\draw (-0.6, 1.5) node[above]{$(y_0^0, u_1^0)$};
\draw (0.8, 1.5) node[above]{$(z_0^0, v_1^0)$};
\draw (2.2, 1.5) node[above]{$(y_0^1, u_1^1)$};
\draw (3.6, 1.5) node[above]{$(z_0^1, v_1^1)$};

\draw[<->, >=Stealth] (-0.1,1.5)--(0.7, 1.5);
\draw[<->, >=Stealth] (2.3,1.5)--(3.1, 1.5);

\fill (1.5, 0) circle (2pt);
\draw (1.5, 0) node[left]{$(w_0, w_1)$};

\draw[->, >=Stealth] (1.5,0.1)--(-0.2, 1.4);
\draw[->, >=Stealth] (1.5,0.1)--(0.8, 1.4);
\draw[->, >=Stealth] (1.5,0.1)--(2.2, 1.4);
\draw[->, >=Stealth] (1.5,0.1)--(3.2, 1.4);
\end{tikzpicture}
\caption{The frame $(W^*, R^*)$ in Case 3}\label{fig:LV,2,1_case3}
\end{figure}
\end{proof}

\subsection{$\Gamma(\LV, \omega, 1)$}

Let $\mathcal{C}_{\Gamma(\LV, \omega, 1)}$ be the class of all finite Kripke models whose frames are of the form shown in Figure \ref{fig:LV,o,1}. 
The logic $\Gamma(\LV, \omega, 1)$ is sound and complete with respect to $\mathcal{C}_{\Gamma(\LV, \omega, 1)}$. 

\begin{figure}[ht]
\centering
\begin{tikzpicture}
\draw (0.3, 1.5) circle (1 and 0.5);
\draw (2.7, 1.5) circle (1 and 0.5);
\fill (-0.4, 1.5) circle (2pt);
\fill (0.2, 1.5) circle (2pt);
\draw (1, 1.5) node{$\cdots$};
\fill (2, 1.5) circle (2pt);
\fill (2.6, 1.5) circle (2pt);
\draw (3.4, 1.5) node{$\cdots$};

\draw[<->, >=Stealth] (-0.3,1.5)--(0.1, 1.5);
\draw[<->, >=Stealth] (0.3,1.5)--(0.7, 1.5);
\draw[<->, >=Stealth] (2.1,1.5)--(2.5, 1.5);
\draw[<->, >=Stealth] (2.7,1.5)--(3.1, 1.5);

\fill (1.5, 0) circle (2pt);

\draw[->, >=Stealth] (1.5,0.1)--(-0.4, 1.4);
\draw[->, >=Stealth] (1.5,0.1)--(0.2, 1.4);
\draw[->, >=Stealth] (1.5,0.1)--(2, 1.4);
\draw[->, >=Stealth] (1.5,0.1)--(2.6, 1.4);
\end{tikzpicture}
\caption{A frame of $\Gamma(\LV, \omega, 1)$}\label{fig:LV,o,1}
\end{figure}
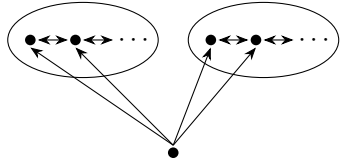

\begin{thm}\label{thm:LV,o,1}
$\Gamma(\LV, \omega, 1)$ has ULIP. 
\end{thm}
\begin{proof}
We prove that $\mathcal{C}_{\Gamma(\LV, \omega, 1)}$ enjoys $3$-IP. 
Let $P^+$ and $P^-$ be any finite sets of propositional variables and let $M_0 = (W_0, R_0, \Vdash_0)$ and $M_1  = (W_1, R_1, \Vdash_1)$ be Kripke models in $\mathcal{C}_{\Gamma(\LV, \omega, 1)}$, let $w_0 \in W_0$ and $w_1 \in W_1$, and suppose $(M_0, w_0) \xr{3} (M_1, w_1)$. 
For $i \in \{0, 1\}$, let $x_i$ be the root element of $M_i$ and let $C_i^0$ and $C_i^1$ be two distinct clusters of final elements of $M_i$. 
We distinguish the following three cases. 

\paragraph*{Case 1:} $w_0 \in C_0^0$ and $w_1 \in C_1^0$. 

\medskip

By Lemma \ref{lem:match_1}, $C_0^0$ matches $C_1^0$. 
Let
\[
    C^* : = \{(u_0^0, u_1^0) \mid u_0^0 \in C_0^0, u_1^0 \in C_1^0,\ \text{and}\ (M_0, u_0^0) \xr{0} (M_1, u_1^0)\}.
\]

\paragraph*{Case 2:} $w_0 \in C_0^0$ and $w_1 = x_1$. 

\medskip

By Lemma \ref{lem:match_2}, $C_0^0$ matches both $C_1^0$ and $C_1^1$. 
Let
\begin{align*}
    C^{*, 0} : & = \{(u_0^0, u_1^0) \mid u_0^0 \in C_0^0, u_1^0 \in C_1^0,\ \text{and}\ (M_0, u_0^0) \xr{0} (M_1, u_1^0)\}, \\
    C^{*, 1} : & = \{(u_0^0, u_1^1) \mid u_0^0 \in C_0^0, u_1^1 \in C_1^1,\ \text{and}\ (M_0, u_0^0) \xr{0} (M_1, u_1^1)\}. 
\end{align*}

\paragraph*{Case 3:} $w_0 = x_0$ and $w_1 = x_1$. 

\medskip

Since both $M_0$ and $M_1$ consist of two clusters, by Lemma \ref{lem:match_two_cl}, there exist two distinct clusters $D_1^0$ and $D_1^1$ of final elements of $M_1$ such that for $k \in \{0, 1\}$, $C_0^k$ matches $D_1^k$. 
Thus, let
\begin{align*}
    C^{*,0} : & = \{(u_0^0, u_1^0) \mid u_0^0 \in C_0^0, u_1^0 \in D_1^0,\ \text{and}\ (M_0, u_0^0) \xr{0} (M_1, u_1^0)\}, \\
    C^{*,1} : & = \{(v_0^0, v_1^0) \mid v_0^0 \in C_0^1, v_1^1 \in D_1^1,\ \text{and}\ (M_0, v_0^0) \xr{0} (M_1, v_1^1)\}. 
\end{align*}

\medskip

Our frames $(W^*, R^*)$ corresponding to these three cases are shown together in Figure \ref{fig:LV,o,1_cases}. 
We can easily check that these frames satisfy all the required conditions.

\begin{figure}[ht]
\centering
\begin{tikzpicture}
\draw (0.3, 1.5) circle (1 and 0.5);
\draw (2.7, 1.5) circle (1 and 0.5);
\fill (-0.4, 1.5) circle (2pt);
\fill (0.2, 1.5) circle (2pt);
\draw (1, 1.5) node{$\cdots$};
\fill (2, 1.5) circle (2pt);
\fill (2.6, 1.5) circle (2pt);
\draw (3.4, 1.5) node{$\cdots$};
\draw (0.3, 2) node[above]{$C^*$};
\draw (2.7, 2) node[above]{$C^*$};

\draw[<->, >=Stealth] (-0.3,1.5)--(0.1, 1.5);
\draw[<->, >=Stealth] (0.3,1.5)--(0.7, 1.5);
\draw[<->, >=Stealth] (2.1,1.5)--(2.5, 1.5);
\draw[<->, >=Stealth] (2.7,1.5)--(3.1, 1.5);

\fill (1.5, 0) circle (2pt);
\draw (1.5, 0) node[left]{$(w_0, w_1)$};
\draw (1.5, -0.5) node[below]{Case 1};

\draw[->, >=Stealth] (1.5,0.1)--(-0.4, 1.4);
\draw[->, >=Stealth] (1.5,0.1)--(0.2, 1.4);
\draw[->, >=Stealth] (1.5,0.1)--(2, 1.4);
\draw[->, >=Stealth] (1.5,0.1)--(2.6, 1.4);

\draw [dashed] (4.5, -1)--(4.5, 2);

\draw (6.3, 1.5) circle (1 and 0.5);
\draw (8.7, 1.5) circle (1 and 0.5);
\fill (5.6, 1.5) circle (2pt);
\fill (6.2, 1.5) circle (2pt);
\draw (7, 1.5) node{$\cdots$};
\fill (8, 1.5) circle (2pt);
\fill (8.6, 1.5) circle (2pt);
\draw (9.4, 1.5) node{$\cdots$};
\draw (6.3, 2) node[above]{$C^{*,0}$};
\draw (8.7, 2) node[above]{$C^{*,1}$};

\draw[<->, >=Stealth] (5.7,1.5)--(6.1, 1.5);
\draw[<->, >=Stealth] (6.3,1.5)--(6.7, 1.5);
\draw[<->, >=Stealth] (8.1,1.5)--(8.5, 1.5);
\draw[<->, >=Stealth] (8.7,1.5)--(9.1, 1.5);

\fill (7.5, 0) circle (2pt);
\draw (7.5, 0) node[left]{$(w_0, w_1)$};
\draw (7.5, -0.5) node[below]{Cases 2 and 3};

\draw[->, >=Stealth] (7.5,0.1)--(5.6, 1.4);
\draw[->, >=Stealth] (7.5,0.1)--(6.2, 1.4);
\draw[->, >=Stealth] (7.5,0.1)--(8, 1.4);
\draw[->, >=Stealth] (7.5,0.1)--(8.6, 1.4);

\end{tikzpicture}
\caption{The frame $(W^*, R^*)$}\label{fig:LV,o,1_cases}
\end{figure}
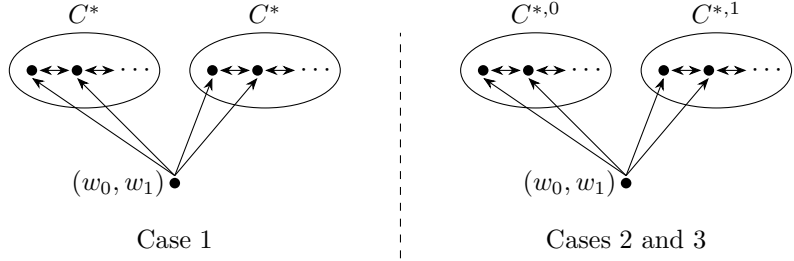
\end{proof}

\section{Modal companions of $\LP_2$}\label{sec:LP2}

In this section, we prove that the three logics $\mathbf{GW} = \Gamma(\LP_2, 1, 1)$, $\Gamma(\LP_2, 2, 1)$, and $\Gamma(\LP_2, \omega, 1)$ have ULIP. 
In particular, LIP for $\Gamma(\LP_2, 2, 1)$ is new.

\subsection{$\mathbf{GW}$}

Let $\mathcal{C}_{\mathbf{GW}}$ be the class of all finite Kripke models whose frames are of the form shown in Figure \ref{fig:GW}. 
The logic $\mathbf{GW}$ is sound and complete with respect to $\mathcal{C}_{\mathbf{GW}}$. 

\begin{figure}[ht]
\centering
\begin{tikzpicture}
\fill (-0.2, 1.5) circle (2pt);
\fill (0.8, 1.5) circle (2pt);
\draw (2.2, 1.5) node{$\cdots$};
\fill (3.2, 1.5) circle (2pt);

\fill (1.5, 0) circle (2pt);

\draw[->, >=Stealth] (1.5,0.1)--(-0.2, 1.4);
\draw[->, >=Stealth] (1.5,0.1)--(0.8, 1.4);
\draw[->, >=Stealth] (1.5,0.1)--(3.2, 1.4);
\end{tikzpicture}
\caption{A frame of $\mathbf{GW}$}\label{fig:GW}
\end{figure}

\begin{thm}[Shimura \cite{Shimura92}; Maksimova \cite{Maksimova14}]\label{thm:GW}
$\mathbf{GW}$ has ULIP. 
\end{thm}
\begin{proof}
We prove that $\mathcal{C}_{\mathbf{GW}}$ enjoys $3$-IP. 
Let $P^+$ and $P^-$ be any finite sets of propositional variables and let $M_0 = (W_0, R_0, \Vdash_0)$ and $M_1  = (W_1, R_1, \Vdash_1)$ be Kripke models in $\mathcal{C}_{\mathbf{GW}}$, let $w_0 \in W_0$ and $w_1 \in W_1$, and suppose $(M_0, w_0) \xr{3} (M_1, w_1)$. 
For $i \in \{0, 1\}$, let $x_i$ be the root element of $M_i$ and let $F_i$ be the set of all final elements of $M_i$. 
We distinguish the following three cases. 

\paragraph*{Case 1:} $w_0 \in F_0$ and $w_1 \in F_1$. 

\medskip

\paragraph*{Case 2:} $w_0 \in F_0$ and $w_1 = x_1$. 

\medskip

By Lemma \ref{lem:match_2}, for all $u_1 \in F_1$, we have that $(M_0, w_0) \xr{0} (M_1, u_1)$. 
Let $F^* : = \{(w_0, u_1) \mid u_1 \in F_1\}$. 

\paragraph*{Case 3:} $w_0 = x_0$ and $w_1 = x_1$. 

\medskip

By Lemma \ref{lem:match_3}, we obtain the following properties:
\begin{enumerate}
    \item For any $u_0 \in F_0$, there exists $u_1 \in F_1$ such that $(M_0, u_0) \xr{0} (M_1, u_1)$. 
    \item For any $u_1 \in F_1$, there exists $u_0 \in F_0$ such that $(M_0, u_0) \xr{0} (M_1, u_1)$. 
\end{enumerate}
Let 
\[
    F^* : = \{(u_0, u_1) \mid u_0 \in F_0, u_1 \in F_1,\ \text{and}\ (M_0, u_0) \xr{0} (M_1, u_1)\}. 
\]

\medskip

Our frames $(W^*, R^*)$ corresponding to these three cases are shown together in Figure \ref{fig:GW_cases}. 
It is easy to check that these frames satisfy all the required conditions.

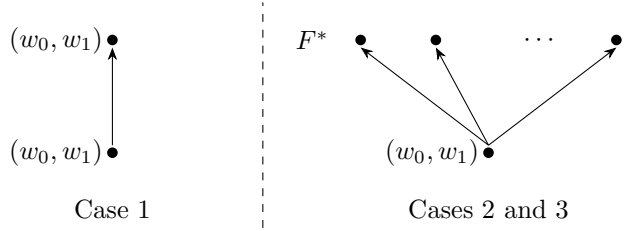
\begin{figure}[ht]
\centering
\begin{tikzpicture}
\fill (-3.5, 1.5) circle (2pt);
\fill (-3.5, 0) circle (2pt);
\draw[->, >=Stealth] (-3.5,0.1)--(-3.5, 1.4);
\draw (-3.5, 0) node[left]{$(w_0, w_1)$}; 
\draw (-3.5, 1.5) node[left]{$(w_0, w_1)$};

\draw [dashed] (-1.5, -1)--(-1.5, 2);

\fill (-0.2, 1.5) circle (2pt);
\fill (0.8, 1.5) circle (2pt);
\draw (2.2, 1.5) node{$\cdots$};
\fill (3.2, 1.5) circle (2pt);
\draw (-3.5, -0.5) node[below]{Case 1};
\draw (-0.5, 1.5) node[left]{$F^*$};

\fill (1.5, 0) circle (2pt);
\draw (1.5, 0) node[left]{$(w_0, w_1)$}; 

\draw[->, >=Stealth] (1.5,0.1)--(-0.2, 1.4);
\draw[->, >=Stealth] (1.5,0.1)--(0.8, 1.4);
\draw[->, >=Stealth] (1.5,0.1)--(3.2, 1.4);
\draw (1.5, -0.5) node[below]{Cases 2 and 3};

\end{tikzpicture}
\caption{The frames $(W^*, R^*)$}\label{fig:GW_cases}
\end{figure}
\end{proof}

\subsection{$\Gamma(\LP_2, 2, 1)$}

Let $\mathcal{C}_{\Gamma(\LP_2, 2, 1)}$ be the class of all finite Kripke models whose frames are of the form shown in Figure \ref{fig:LP2,2,1}. 
The logic $\Gamma(\LP_2, 2, 1)$ is sound and complete with respect to $\mathcal{C}_{\Gamma(\LP_2, 2, 1)}$. 

\begin{figure}[ht]
\centering
\begin{tikzpicture}
\fill (-1, 1.5) circle (2pt);
\fill (-0.2, 1.5) circle (2pt);
\fill (0.6, 1.5) circle (2pt);
\fill (1.4, 1.5) circle (2pt);
\draw (2.2, 1.5) node{$\cdots$};
\fill (3, 1.5) circle (2pt);
\fill (3.8, 1.5) circle (2pt);

\draw[<->, >=Stealth] (-0.9,1.5)--(-0.1, 1.5);
\draw[<->, >=Stealth] (0.7,1.5)--(1.3, 1.5);
\draw[<->, >=Stealth] (3.1,1.5)--(3.7, 1.5);

\fill (1.4, 0) circle (2pt);

\draw[->, >=Stealth] (1.4,0.1)--(-1, 1.4);
\draw[->, >=Stealth] (1.4,0.1)--(-0.2, 1.4);
\draw[->, >=Stealth] (1.4,0.1)--(0.6, 1.4);
\draw[->, >=Stealth] (1.4,0.1)--(1.4, 1.4);
\draw[->, >=Stealth] (1.4,0.1)--(3, 1.4);
\draw[->, >=Stealth] (1.4,0.1)--(3.8, 1.4);
\end{tikzpicture}
\caption{A frame of $\Gamma(\LP_2, 2, 1)$}\label{fig:LP2,2,1}
\end{figure}
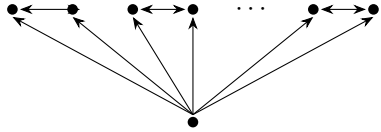

\begin{thm}\label{thm:LP2,2,1}
$\Gamma(\LP_2, 2, 1)$ has ULIP. 
\end{thm}
\begin{proof}
We prove that $\mathcal{C}_{\Gamma(\LP_2, 2, 1)}$ enjoys $3$-IP. 
Let $P^+$ and $P^-$ be any finite sets of propositional variables and let $M_0 = (W_0, R_0, \Vdash_0)$ and $M_1  = (W_1, R_1, \Vdash_1)$ be Kripke models in $\mathcal{C}_{\Gamma(\LP_2, 2, 1)}$, let $w_0 \in W_0$ and $w_1 \in W_1$, and suppose $(M_0, w_0) \xr{3} (M_1, w_1)$. 
For $i \in \{0, 1\}$, let $x_i$ be the root element of $M_i$ and let $C_i^0, C_i^1, \ldots, C_i^{j_i}$ be all the clusters of final elements of $M_i$. 
Also let $C_i^k = \{y_i^k, z_i^k\}$. 
We distinguish the following three cases. 

\paragraph*{Case 1:} $w_0 \in C_0^0$ and $w_1 \in C_1^0$. 

\medskip

By Lemma \ref{lem:match_1}, $C_0^0$ matches $C_1^0$. 
Since these clusters consist of two elements, by Lemma \ref{lem:match_two}, there exist $u_1^0, v_1^0 \in C_1^0$ such that $(M_0, y_0^0) \xr{0} (M_1, u_1^0)$ and $(M_0, z_0^0) \xr{0} (M_1, v_1^0)$. 

\paragraph*{Case 2:} $w_0 \in C_0^0$ and $w_1 = x_1$. 

\medskip

By Lemma \ref{lem:match_2}, $C_0^0$ matches all the clusters $C_1^k$ of final elements of $M_1$. 
Since every cluster consist of two elements, by Lemma \ref{lem:match_two}, for each cluster $C_1^k$ of $M_1$, there exist $u_1^k, v_1^k \in C_1^k$ such that $(M_0, y_0^0) \xr{0} (M_1, u_1^k)$ and $(M_0, z_0) \xr{0} (M_1, v_1^k)$.

\paragraph*{Case 3:} $w_0 = x_0$ and $w_1 = x_1$. 

\medskip

By Lemmas \ref{lem:match_3} and \ref{lem:match_two}, we obtain the following two properties: 
\begin{enumerate}
    \item For any cluster $C_0^k$ of $M_0$, there exist a cluster $C_1^{k'}$ of $M_1$ and $u_1^{k'}, v_1^{k'} \in C_1^{k'}$ such that $u_1^{k'} \neq v_1^{k'}$ and  
\[
    (M_0, y_0^k) \xr{0} (M_1, u_1^{k'})\ \text{and}\ (M_0, z_0^k) \xr{0} (M_1, v_1^{k'}). 
\]

    \item For any cluster $C_1^k$ of $M_1$, there exist a cluster $C_0^{k'}$ of $M_0$ and $u_0^{k'}, v_0^{k'} \in C_0^{k'}$ such that $u_0^{k'} \neq v_0^{k'}$ and  
\[
    (M_0, u_0^{k'}) \xr{0} (M_1, y_1^{k})\ \text{and}\ (M_0, v_0^{k'}) \xr{0} (M_1, z_1^{k}). 
\]
\end{enumerate}

\medskip

Our frames $(W^*, R^*)$ corresponding to Cases 1 and 2 are drawn in Figure \ref{fig:LP2,2,1_case12}. 
Our frame $(W^*, R^*)$ corresponding to Case 3 is shown in Figure \ref{fig:LP2,2,1_case3}.
These frames satisfy all the required conditions.

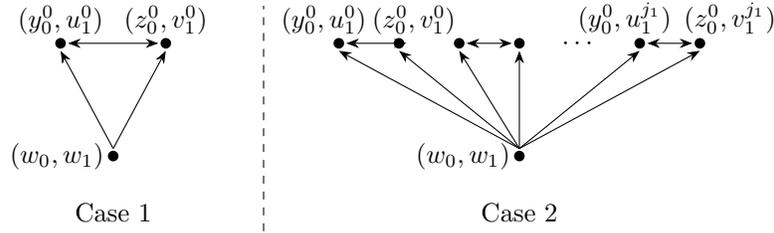
\begin{figure}[ht]
\centering
\begin{tikzpicture}
\fill (-3.3, 1.5) circle (2pt);
\fill (-4.7, 1.5) circle (2pt);
\fill (-4, 0) circle (2pt);
\draw (-4.7, 1.5) node[above]{$(y_0^0, u_1^0)$};
\draw (-3.3, 1.5) node[above]{$(z_0^0, v_1^0)$};
\draw (-4, 0) node[left]{$(w_0, w_1)$};

\draw[->, >=Stealth] (-4,0.1)--(-4.7, 1.4);
\draw[->, >=Stealth] (-4,0.1)--(-3.3, 1.4);
\draw[<->, >=Stealth] (-4.6,1.5)--(-3.4, 1.5);
\draw (-4, -0.5) node[below]{Case 1};

\draw[dashed] (-2, -1) -- (-2, 2);

\fill (-1, 1.5) circle (2pt);
\fill (-0.2, 1.5) circle (2pt);
\fill (0.6, 1.5) circle (2pt);
\fill (1.4, 1.5) circle (2pt);
\draw (2.2, 1.5) node{$\cdots$};
\fill (3, 1.5) circle (2pt);
\fill (3.8, 1.5) circle (2pt);
\draw (1.4, 0) node[left]{$(w_0, w_1)$};
\draw (-1.2, 1.5) node[above]{$(y_0^0, u_1^0)$};
\draw (0, 1.5) node[above]{$(z_0^0, v_1^0)$};
\draw (2.8, 1.5) node[above]{$(y_0^0, u_1^{j_1})$};
\draw (4.2, 1.5) node[above]{$(z_0^0, v_1^{j_1})$};

\draw[<->, >=Stealth] (-0.9,1.5)--(-0.1, 1.5);
\draw[<->, >=Stealth] (0.7,1.5)--(1.3, 1.5);
\draw[<->, >=Stealth] (3.1,1.5)--(3.7, 1.5);

\fill (1.4, 0) circle (2pt);

\draw[->, >=Stealth] (1.4,0.1)--(-1, 1.4);
\draw[->, >=Stealth] (1.4,0.1)--(-0.2, 1.4);
\draw[->, >=Stealth] (1.4,0.1)--(0.6, 1.4);
\draw[->, >=Stealth] (1.4,0.1)--(1.4, 1.4);
\draw[->, >=Stealth] (1.4,0.1)--(3, 1.4);
\draw[->, >=Stealth] (1.4,0.1)--(3.8, 1.4);
\draw (1.4, -0.5) node[below]{Case 2};
\end{tikzpicture}
\caption{The frames $(W^*, R^*)$ in Cases 1 and 2}\label{fig:LP2,2,1_case12}
\end{figure}

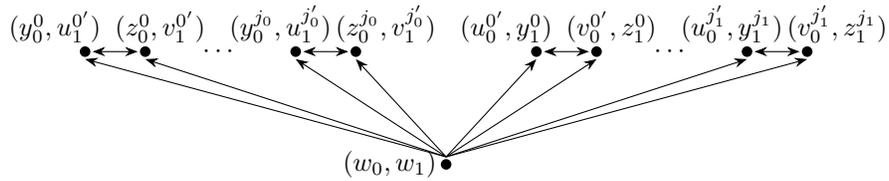
\begin{figure}[ht]
\centering
\begin{tikzpicture}
\fill (-4.8, 1.5) circle (2pt);
\fill (-4, 1.5) circle (2pt);
\draw (-3, 1.5) node{$\cdots$};
\fill (-2, 1.5) circle (2pt);
\fill (-1.2, 1.5) circle (2pt);
\fill (4.8, 1.5) circle (2pt);
\fill (4, 1.5) circle (2pt);
\draw (3, 1.5) node{$\cdots$};
\fill (2, 1.5) circle (2pt);
\fill (1.2, 1.5) circle (2pt);
\draw (0, 0) node[left]{$(w_0, w_1)$};
\draw (-5.2, 1.5) node[above]{$(y_0^0, u_1^{0'})$};
\draw (-3.8, 1.5) node[above]{$(z_0^0, v_1^{0'})$};
\draw (-2.2, 1.5) node[above]{$(y_0^{j_0}, u_1^{j_0'})$};
\draw (-0.8, 1.5) node[above]{$(z_0^{j_0}, v_1^{j_0'})$};

\draw (5.2, 1.5) node[above]{$(v_0^{j_1'}, z_1^{j_1})$};
\draw (3.8, 1.5) node[above]{$(u_0^{j_1'}, y_1^{j_1})$};
\draw (2.2, 1.5) node[above]{$(v_0^{0'}, z_1^{0})$};
\draw (0.8, 1.5) node[above]{$(u_0^{0'}, y_1^{0})$};

\draw[<->, >=Stealth] (-4.7,1.5)--(-4.1, 1.5);
\draw[<->, >=Stealth] (-1.9,1.5)--(-1.3, 1.5);
\draw[<->, >=Stealth] (4.7,1.5)--(4.1, 1.5);
\draw[<->, >=Stealth] (1.9,1.5)--(1.3, 1.5);

\fill (0, 0) circle (2pt);

\draw[->, >=Stealth] (0,0.1)--(-4.8, 1.4);
\draw[->, >=Stealth] (0,0.1)--(-4, 1.4);
\draw[->, >=Stealth] (0,0.1)--(-2, 1.4);
\draw[->, >=Stealth] (0,0.1)--(-1.2, 1.4);
\draw[->, >=Stealth] (0,0.1)--(4.8, 1.4);
\draw[->, >=Stealth] (0,0.1)--(4, 1.4);
\draw[->, >=Stealth] (0,0.1)--(2, 1.4);
\draw[->, >=Stealth] (0,0.1)--(1.2, 1.4);
\end{tikzpicture}
\caption{The frame $(W^*, R^*)$ in Case 3}\label{fig:LP2,2,1_case3}
\end{figure}
\end{proof}

\subsection{$\Gamma(\LP_2, \omega, 1)$}

Let $\mathcal{C}_{\Gamma(\LP_2, \omega, 1)}$ be the class of all finite Kripke models whose frames are of the form shown in Figure \ref{fig:LP2,o,1}. 
The logic $\Gamma(\LP_2, \omega, 1)$ is sound and complete with respect to $\mathcal{C}_{\Gamma(\LP_2, \omega, 1)}$. 

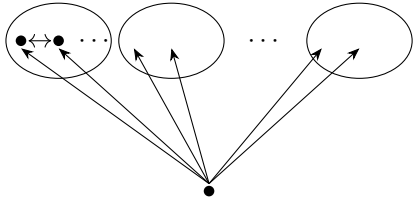
\begin{figure}[ht]
\centering
\begin{tikzpicture}
\draw (-2, 2) circle (0.7 and 0.5);
\draw (-0.5, 2) circle (0.7 and 0.5);
\draw (2, 2) circle (0.7 and 0.5);
\draw (0.75, 2) node{$\cdots$};

\fill (-2.5, 2) circle (2pt);
\fill (-2, 2) circle (2pt);
\draw (-1.5, 2) node{$\cdots$};

\fill (0, 0) circle (2pt);

\draw[<->] (-2.4,2)--(-2.1, 2);
\draw[->, >=Stealth] (0,0.1)--(-2.5, 1.9);
\draw[->, >=Stealth] (0,0.1)--(-2, 1.9);
\draw[->, >=Stealth] (0,0.1)--(-0.5, 1.9);
\draw[->, >=Stealth] (0,0.1)--(-1, 1.9);
\draw[->, >=Stealth] (0,0.1)--(1.5, 1.9);
\draw[->, >=Stealth] (0,0.1)--(2, 1.9);
\end{tikzpicture}
\caption{A frame of $\Gamma(\LP_2, \omega, 1)$}\label{fig:LP2,o,1}
\end{figure}

\begin{thm}[Shimura \cite{Shimura92}]\label{thm:LP2,o,1}
$\Gamma(\LP_2, \omega, 1)$ has ULIP. 
\end{thm}
\begin{proof}
We prove that $\mathcal{C}_{\Gamma(\LP_2, \omega, 1)}$ enjoys $3$-IP. 
Let $P^+$ and $P^-$ be any finite sets of propositional variables and let $M_0 = (W_0, R_0, \Vdash_0)$ and $M_1  = (W_1, R_1, \Vdash_1)$ be Kripke models in $\mathcal{C}_{\Gamma(\LP_2, \omega, 1)}$, let $w_0 \in W_0$ and $w_1 \in W_1$, and suppose $(M_0, w_0) \xr{3} (M_1, w_1)$. 
For $i \in \{0, 1\}$, let $x_i$ be the root element of $M_i$ and let $C_i^0, C_i^1, \ldots, C_i^{j_i}$ be all the clusters of final elements of $M_i$. 
We distinguish the following three cases. 

\paragraph*{Case 1:} $w_0 \in C_0^0$ and $w_1 \in C_1^0$. 

\medskip

By Lemma \ref{lem:match_1}, $C_0^0$ matches $C_1^0$. 
Let 
\[
    C^* : = \{(u_0^0, u_1^0) \mid u_0^0 \in C_0^0, u_1^0 \in C_1^0,\ \text{and}\ (M_0, u_0^0) \xr{0} (M_1, u_1^0)\}.
\]

\paragraph*{Case 2:} $w_0 \in C_0^0$ and $w_1 = x_1$. 

\medskip

By Lemma \ref{lem:match_2}, $C_0^0$ matches all the clusters $C_1^k$ of final elements of $M_1$. 
Then, for each cluster $C_1^k$ of $M_1$, let 
\[
    C^{*,k} = \{(u_0^0, u_1^k) \mid u_0^0 \in C_0^0, u_1^k \in C_1^k, \ \text{and}\ (M_0, u_0^0) \xr{0} (M_1, u_1^k)\}.
\]

\paragraph*{Case 3:} $w_0 = x_0$ and $w_1 = x_1$. 

\medskip

By Lemma \ref{lem:match_3}, we obtain the following two properties: 
\begin{enumerate}
    \item For any cluster $C_0^k$ of $M_0$, there exists a cluster $C_1^{k'}$ of $M_1$ such that $C_0^k$ matches $C_1^{k'}$. 
    \item For any cluster $C_1^k$ of $M_1$, there exists a cluster $C_0^{k'}$ of $M_0$ such that $C_0^{k'}$ matches $C_1^{k}$. 
\end{enumerate}
For each cluster $C_0^k$ of $M_0$, let 
\[
    C_0^{*, k} : = \{(u_0^k, u_1^{k'}) \mid u_0^k \in C_0^k, u_1^{k'} \in C_1^{k'}, \ \text{and}\ (M_0, u_0^k) \xr{0} (M_1, u_1^{k'})\}.
\]
Also, for each cluster $C_1^k$ of $M_1$, let 
\[
    C_1^{*, k} : = \{(u_0^{k'}, u_1^{k}) \mid u_0^{k'} \in C_0^{k'}, u_1^{k} \in C_1^{k}, \ \text{and}\ (M_0, u_0^{k'}) \xr{0} (M_1, u_1^{k})\}.
\]

\medskip

Our frames $(W^*, R^*)$ corresponding to Cases 1 and 2 are drawn in Figure \ref{fig:LP2,o,1_case12}. 
Our frame $(W^*, R^*)$ corresponding to Case 3 is shown in Figure \ref{fig:LP2,o,1_case3}.
We can check that these frames satisfy all the required conditions.

\begin{figure}[ht]
\centering
\begin{tikzpicture}
\draw (-6, 2) circle (0.7 and 0.5);
\fill (-6, 0) circle (2pt);
\draw (-6, 2.5) node[above]{$C^*$};

\draw (-6, -0.5) node[below]{Case 1};

\fill (-6.5, 2) circle (2pt);
\fill (-6, 2) circle (2pt);
\draw (-5.5, 2) node{$\cdots$};

\draw[<->] (-6.4,2)--(-6.1, 2);
\draw[->, >=Stealth] (-6,0.1)--(-6.5, 1.9);
\draw[->, >=Stealth] (-6,0.1)--(-6, 1.9);

\draw[->, >=Stealth] (-6,0.1)--(-6, 1.9);
\draw[->, >=Stealth] (-6,0.1)--(-6.5, 1.9);

\draw[dashed] (-4, -1) -- (-4, 2.5);

\draw (-2, 2) circle (0.7 and 0.5);
\draw (-0.5, 2) circle (0.7 and 0.5);
\draw (2, 2) circle (0.7 and 0.5);
\draw (0.75, 2) node{$\cdots$};
\draw (-2, 2.5) node[above]{$C^{*,0}$};
\draw (-0.5, 2.5) node[above]{$C^{*,1}$};
\draw (2, 2.5) node[above]{$C^{*,j_1}$};

\fill (-2.5, 2) circle (2pt);
\fill (-2, 2) circle (2pt);
\draw (-1.5, 2) node{$\cdots$};

\fill (0, 0) circle (2pt);

\draw[<->] (-2.4,2)--(-2.1, 2);
\draw[->, >=Stealth] (0,0.1)--(-2.5, 1.9);
\draw[->, >=Stealth] (0,0.1)--(-2, 1.9);
\draw[->, >=Stealth] (0,0.1)--(-0.5, 1.9);
\draw[->, >=Stealth] (0,0.1)--(-1, 1.9);
\draw[->, >=Stealth] (0,0.1)--(1.5, 1.9);
\draw[->, >=Stealth] (0,0.1)--(2, 1.9);
\draw (0, -0.5) node[below]{Case 2};
\end{tikzpicture}
\caption{The frames $(W^*, R^*)$ in Cases 1 and 2}\label{fig:LP2,o,1_case12}
\end{figure}
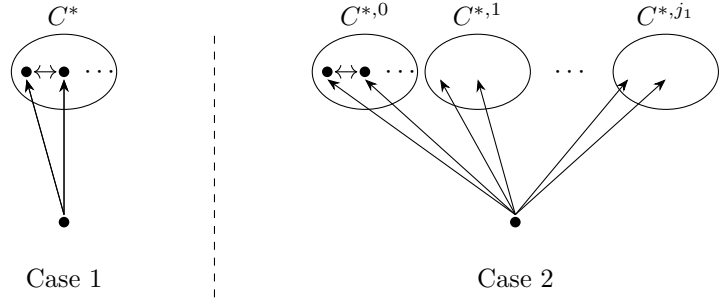

\begin{figure}[ht]
\centering
\begin{tikzpicture}
\draw (-4.4, 2) circle (0.7 and 0.5);
\draw (-3, 2) node{$\cdots$};
\draw (-1.6, 2) circle (0.7 and 0.5);
\draw (4.4, 2) circle (0.7 and 0.5);
\draw (3, 2) node{$\cdots$};
\draw (1.6, 2) circle (0.7 and 0.5);
\draw (0, 0) node[left]{$(w_0, w_1)$};
\draw (-4.4, 2.5) node[above]{$C_0^{*, 0}$};
\draw (-1.6, 2.5) node[above]{$C_0^{*, j_0}$};
\draw (1.6, 2.5) node[above]{$C_1^{*, 0}$};
\draw (4.4, 2.5) node[above]{$C_1^{*, j_1}$};

\fill (0, 0) circle (2pt);

\draw[->, >=Stealth] (0,0.1)--(-4.8, 2);
\draw[->, >=Stealth] (0,0.1)--(-4, 2);
\draw[->, >=Stealth] (0,0.1)--(-2, 2);
\draw[->, >=Stealth] (0,0.1)--(-1.2, 2);
\draw[->, >=Stealth] (0,0.1)--(4.8, 2);
\draw[->, >=Stealth] (0,0.1)--(4, 2);
\draw[->, >=Stealth] (0,0.1)--(2, 2);
\draw[->, >=Stealth] (0,0.1)--(1.2, 2);
\end{tikzpicture}
\caption{The frame $(W^*, R^*)$ in Case 3}\label{fig:LP2,o,1_case3}
\end{figure}
\end{proof}

\section{Failure of LIP}\label{sec:failure}

In this section, we investigate the failure of LIP. 
Maksimova \cite{Maksimova82} proved that each of the logics $\Gamma(\LP_2, 1, 2)$, $\Gamma(\LV, 1, 2)$, $\Gamma(\LS, 1, 2)$, and $\Gamma(\Cl, 2, 0)$ has CIP but does not have LIP (see also \cite{Maksimova91,GM05}). 
We prove that the three further logics $\Gamma(\LP_2, 1, \omega)$, $\Gamma(\LV, 1, \omega)$, and $\Gamma(\LS, 1, \omega)$ also do not have LIP. 
Moreover, we will prove this in a form that includes a part of Maksimova's results, namely, we prove that every logic $L$ satisfying $\Gamma(\LP_2, 1, \omega) \subseteq L \subseteq \Gamma(\LS, 1, 2)$ does not have LIP. 

\begin{figure}[ht]
\centering
\begin{tikzpicture}
\draw (0, 0) node{$\Gamma(\LP_2, 1, \omega)$};
\draw (1.5, 0) node{$\subseteq$};
\draw (3, 0) node{$\Gamma(\LV, 1, \omega)$};
\draw (4.5, 0) node{$\subseteq$};
\draw (6, 0) node{$\Gamma(\LS, 1, \omega)$};
\draw (1, 1.5) node{$\Gamma(\LP_2, 1, 2)$};
\draw (2.5, 1.5) node{$\subseteq$};
\draw (4, 1.5) node{$\Gamma(\LV, 1, 2)$};
\draw (5.5, 1.5) node{$\subseteq$};
\draw (7, 1.5) node{$\Gamma(\LS, 1, 2)$};
\draw (0.8, 0.8) node{\rotatebox{45}{$\subseteq$}};
\draw (3.8, 0.8) node{\rotatebox{45}{$\subseteq$}};
\draw (6.8, 0.8) node{\rotatebox{45}{$\subseteq$}};
\end{tikzpicture}
\end{figure}

Before proving our theorem, we prepare the following lemma. 

\begin{lem}\label{eight}
Over the logic $\Gamma(\LS, 1, 2)$, every formula $\varphi$ satisfying $v^+(\varphi) \subseteq \{p\}$ and $v^-(\varphi) = \emptyset$ is provably equivalent to one of the following eight formulas: 
\[
    \bot, \quad \Box p,  \quad p \land \Box \Diamond p,  \quad \Box \Diamond p,  \quad p,  \quad p \lor \Box \Diamond p,  \quad \Diamond p,  \quad \top.
\]
\end{lem}
\begin{proof}
The logic $\Gamma(\LS, 1, 2)$ is characterized by the Kripke frame shown in Figure \ref{figure_model0}. 
It is easy to see that $\Gamma(\LS, 1, 2)$ contains $\Box \Diamond p \leftrightarrow \Diamond \Box p$. 

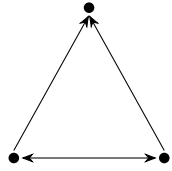
\begin{figure}[ht]
\centering
\begin{tikzpicture}
\fill (0, 0) circle (2pt);
\fill (-1, -2) circle (2pt);
\fill (1, -2) circle (2pt);

\draw[<->, >=Stealth] (-0.9,-2)--(0.9, -2);
\draw[->, >=Stealth] (-1,-1.9)--(0, -0.1);
\draw[->, >=Stealth] (1,-1.9)--(0, -0.1);
\end{tikzpicture}
\caption{The frame of $\Gamma(\LS, 1, 2)$}\label{figure_model0}
\end{figure}

The diagram in Figure \ref{fig:implications} shows the implications between these eight formulas over $\SF$.

\begin{figure}[ht]
\centering
\begin{tikzpicture}
\node (bot) at (0,0) {$\bot$};
\node (Bp) at (1.5, 0) {$\Box p$};
\node (p_BDp) at (3, 0) {$p \land \Box \Diamond p$};
\node (p) at (4.5, 1) {$p$};
\node (BDp) at (4.5, -1) {$\Box \Diamond p$};
\node (pvBDp) at (6, 0) {$p \lor \Box \Diamond p$};
\node (Dp) at (7.5, 0) {$\Diamond p$};
\node (top) at (9, 0) {$\top$};

\draw[->, >=Stealth] (bot)--(Bp);
\draw[->, >=Stealth] (Bp)--(p_BDp);
\draw[->, >=Stealth] (p_BDp)--(p);
\draw[->, >=Stealth] (p_BDp)--(BDp);
\draw[->, >=Stealth] (p)--(pvBDp);
\draw[->, >=Stealth] (BDp)--(pvBDp);
\draw[->, >=Stealth] (pvBDp)--(Dp);
\draw[->, >=Stealth] (Dp)--(top);

\end{tikzpicture}
\caption{The implication between the eight formulas over $\SF$}\label{fig:implications}
\end{figure}
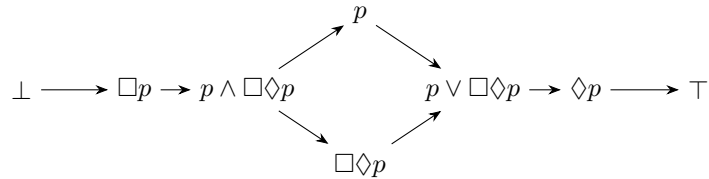

We prove the lemma by induction on the construction of $\varphi$.
If $\varphi$ is one of $p$, $\bot$, and $\top$, then the lemma is trivial. 

Suppose that the lemma holds for $\varphi_0$ and $\varphi_1$. 
It is easy to show that $\varphi_0 \land \varphi_1$ and $\varphi_0 \lor \varphi_1$ are equivalent to one of the eight formulas. 
The following list of equivalences shows that $\Box \varphi_0$ and $\Diamond \varphi_0$ are also equivalent to some of the eight formulas. 

\begin{tabbing}
\hspace{50mm} \= \hspace{50mm} \kill
$\Box \bot \leftrightarrow \bot$ \> $\Diamond \bot \leftrightarrow \bot$ \\
$\Box \Box p \leftrightarrow \Box p$ \> $\Diamond \Box p \leftrightarrow \Box \Diamond p$ \\
$\Box (p \land \Box \Diamond p) \leftrightarrow \Box p$ \> $\Diamond (p \land \Box \Diamond p) \leftrightarrow \Box \Diamond p$ \\
$\Box p \leftrightarrow \Box p$ \> $\Diamond p \leftrightarrow \Diamond p$ \\
$\Box \Box \Diamond p \leftrightarrow \Box \Diamond p$ \> $\Diamond \Box \Diamond p \leftrightarrow \Box \Diamond p$ \\
$\Box (p \lor \Box \Diamond p) \leftrightarrow \Box \Diamond p$ \> $\Diamond (p \lor \Box \Diamond p) \leftrightarrow \Diamond p$ \\
$\Box \Diamond p \leftrightarrow \Box \Diamond p$ \> $\Diamond \Diamond p \leftrightarrow \Diamond p$ \\
$\Box \top \leftrightarrow \top$ \> $\Diamond \top \leftrightarrow \top$
\end{tabbing}
\end{proof}

Here, we are ready to prove the theorem. 

\begin{thm}\label{failure}
Let $L$ be any logic satisfying $\Gamma(\LP_2, 1, \omega) \subseteq L \subseteq \Gamma(\LS, 1, 2)$. 
Then, $L$ does not have LIP. 
In particular, each of the logics $\Gamma(\LP_2, 1, \omega)$, $\Gamma(\LP_2, 1, 2)$, $\Gamma(\LV, 1, \omega)$, $\Gamma(\LV, 1, 2)$, $\Gamma(\LS, 1, \omega)$, and $\Gamma(\LS, 1, 2)$ has CIP but does not have LIP.
\end{thm}
\begin{proof}
At first, we prove
\[
    \Gamma(\LP_2, 1, \omega) \vdash p \land \Box(\Box \neg p \lor p) \to \Box (p \lor q \lor \Box \neg q).
\]

Let $(W, R, \Vdash)$ be any finite Kripke model whose frame is of the form shown in Figure \ref{fig:LP2,1,o}. 
The logic $\Gamma(\LP_2, 1, \omega)$ is characterized by the class of all such Kripke frames. 

\begin{figure}[ht]
\centering
\begin{tikzpicture}
\fill (-0.5, 0) circle (2pt);
\fill (0.5, 0) circle (2pt);
\draw (1.5, 0) node{$\cdots$};
\fill (2.5, 0) circle (2pt);
\fill (3.5, 0) circle (2pt);
\draw (1.5, -1.5) circle (2 and 0.5);
\draw (4, 0) node[right]{$F$};

\fill (0, -1.5) circle (2pt);
\fill (1, -1.5) circle (2pt);
\draw (2, -1.5) node{$\cdots$};
\fill (3, -1.5) circle (2pt);
\draw (4, -1.5) node[right]{$I$};

\draw[<->, >=Stealth] (0.1,-1.5)--(0.9, -1.5);
\draw[<->, >=Stealth] (1.1,-1.5)--(1.6, -1.5);
\draw[<->, >=Stealth] (2.4,-1.5)--(2.9, -1.5);
\draw[->, >=Stealth] (0,-1.4)--(-0.5, -0.1);
\draw[->, >=Stealth] (1,-1.4)--(-0.5, -0.1);
\draw[->, >=Stealth] (3,-1.4)--(-0.5, -0.1);
\draw[->, >=Stealth] (0,-1.4)--(0.5, -0.1);
\draw[->, >=Stealth] (1,-1.4)--(0.5, -0.1);
\draw[->, >=Stealth] (3,-1.4)--(0.5, -0.1);
\draw[->, >=Stealth] (0,-1.4)--(2.5, -0.1);
\draw[->, >=Stealth] (1,-1.4)--(2.5, -0.1);
\draw[->, >=Stealth] (3,-1.4)--(2.5, -0.1);
\draw[->, >=Stealth] (0,-1.4)--(3.5, -0.1);
\draw[->, >=Stealth] (1,-1.4)--(3.5, -0.1);
\draw[->, >=Stealth] (3,-1.4)--(3.5, -0.1);

\end{tikzpicture}
\caption{A frame of $\Gamma(\LP_2, 1, \omega)$}\label{fig:LP2,1,o}
\end{figure}
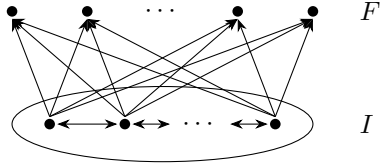

Let $F$ be the set of all final elements of the model, and $I$ be the cluster $W \setminus F$ of inner elements. 
It suffices to show that $p \land \Box(\Box \neg p \lor p) \to \Box (p \lor q \lor \Box \neg q)$ is valid in $(W, R, \Vdash)$. 
For $x \in F$, we have $x \Vdash p \land \Box(\Box \neg p \lor p) \to \Box (p \lor q \lor \Box \neg q)$ because $x \Vdash p \to \Box p$. 
For $x \in I$, assume $x \Vdash p \land \Box (\Box \neg p \lor p)$. 
Let $y \in W$ be any element such that $x R y$. 
If $y \in I$, then $y \Vdash \Diamond p \land (\Box \neg p \lor p)$, and so $y \Vdash p$. 
If $y \in F$, then $y \Vdash q \lor \Box \neg q$ holds. 
Hence, we obtain $x \Vdash \Box (p \lor q \lor \Box \neg q)$. 
We have shown that $p \land \Box(\Box \neg p \lor p) \to \Box (p \lor q \lor \Box \neg q)$ is valid in $(W, R, \Vdash)$. 

Suppose, towards a contradiction, that the logic $L$ has LIP. 
Then, there would exist a Lyndon interpolant $\theta$ of $p \land \Box(\Box \neg p \lor p) \to \Box (p \lor q \lor \Box \neg q)$. 
Since $L \subseteq \Gamma(\LS, 1, 2)$, we would have: 
\begin{enumerate}
    \item $\Gamma(\LS, 1, 2) \vdash p \land \Box(\Box \neg p \lor p) \to \theta$.
    \item $\Gamma(\LS, 1, 2) \vdash \theta \to \Box(p \lor q \lor \Box \neg q)$. 
    \item $v^+(\theta) \subseteq \{p\}$ and $v^-(\theta) = \emptyset$. 
\end{enumerate}

Here, we consider the model $M_0 = (W_0, R_0, \Vdash_0)$ of $\Gamma(\LS, 1, 2)$ as drawn in Figure \ref{figure_model1}: 
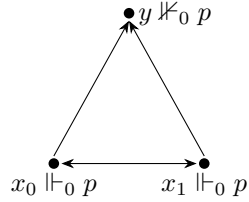
\begin{figure}[ht]
\centering
\begin{tikzpicture}
\fill (0, 0) circle (2pt);
\fill (-1, -2) circle (2pt);
\fill (1, -2) circle (2pt);

\draw[<->, >=Stealth] (-0.9,-2)--(0.9, -2);
\draw[->, >=Stealth] (-1,-1.9)--(0, -0.1);
\draw[->, >=Stealth] (1,-1.9)--(0, -0.1);

\draw (0, 0) node[right]{$y \nVdash_0 p$};
\draw (-1, -2) node[below]{$x_0 \Vdash_0 p$};
\draw (1, -2) node[below]{$x_1 \Vdash_0 p$};
\end{tikzpicture}
\caption{The model $M_0$}\label{figure_model1}
\end{figure}

It is easy to see that $(M_0, x_0) \Vdash_0 p \land \Box (\Box \neg p \lor p)$. 
Thus, we have $(M_0, x_0) \Vdash_0 \theta$ because $\Gamma(\LS, 1, 2) \vdash p \land \Box(\Box \neg p \lor p) \to \theta$. 
On the other hand, since $(M_0, x_0) \nVdash_0 \Box \Diamond p$, we have that $\theta$ is equivalent to one of $p$, $p \lor \Box \Diamond p$, $\Diamond p$, and $\top$ over $\Gamma(\LS, 1, 2)$ by Lemma \ref{eight}. 
So, $\Gamma(\LS, 1, 2) \vdash p \to \theta$. 

Next, we consider the model $M_1 = (W_0, R_0, \Vdash_1)$ of $\Gamma(\LS, 1,2)$ as drawn in Figure \ref{figure_model2}: 

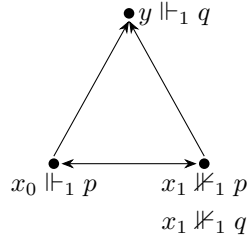
\begin{figure}[ht]
\centering
\begin{tikzpicture}
\fill (0, 0) circle (2pt);
\fill (-1, -2) circle (2pt);
\fill (1, -2) circle (2pt);

\draw[<->, >=Stealth] (-0.9,-2)--(0.9, -2);
\draw[->, >=Stealth] (-1,-1.9)--(0, -0.1);
\draw[->, >=Stealth] (1,-1.9)--(0, -0.1);

\draw (0, 0) node[right]{$y \Vdash_1 q$};
\draw (-1, -2) node[below]{$x_0 \Vdash_1 p$};
\draw (1, -2) node[below]{$x_1 \nVdash_1 p$};
\draw (1, -2.5) node[below]{$x_1 \nVdash_1 q$};
\end{tikzpicture}
\caption{The model $M_1$}\label{figure_model2}
\end{figure}

Since $(M_1, x_0) \Vdash_1 p$, we have $(M_1, x_0) \Vdash_1 \theta$. 
On the other hand, $(M_1, x_1) \nVdash_1 p$, $(M_1, x_1) \nVdash_1 q$, and $(M_1, x_1) \nVdash_1 \Box \neg q$ imply that $(M_1, x_0) \nVdash_1 \Box(p \lor q \lor \Box \neg q)$. 
This contradicts $\Gamma(\LS, 1, 2) \vdash \theta \to \Box(p \lor q \lor \Box \neg q)$. 

Therefore, $L$ does not have LIP. 
\end{proof}

\section{LIP and ULIP for intermediate propositional logics}\label{sec:intermediate}

In this section, we discuss LIP and ULIP for intermediate propositional logics. 
In particular, we will prove the following theorem. 

\begin{thm}\label{thm:intermediate}
For any consistent intermediate propositional logic $L$, the following are equivalent: 
\begin{enumerate}
    \item $L$ has CIP. 
    \item $L$ has UIP. 
    \item $L$ has LIP. 
    \item $L$ has ULIP. 
\end{enumerate}
\end{thm}
The equivalence $(1 \Leftrightarrow 2)$ is already known \cite{Maksimova14}. 
So, we will prove the equivalences $(1 \Leftrightarrow 3)$ and $(1 \Leftrightarrow 4)$. 

First, we consider LIP. 
LIP for intermediate logics directly follows from LIP for modal logics through G\"odel's translation.

\begin{defn}[G\"odel's translation]
We define the translation $\mathsf{T}$ from formulas of intermediate logic to formulas of modal logic as follows:
\begin{itemize}
    \item $\mathsf{T}(p) = \Box p$ for propositional variables $p$, 
    \item $\mathsf{T}(\bot) = \bot$, 
    \item $\mathsf{T}(\varphi \land \psi) = \mathsf{T}(\varphi) \land \mathsf{T}(\psi)$, 
    \item $\mathsf{T}(\varphi \lor \psi) = \mathsf{T}(\varphi) \lor \mathsf{T}(\psi)$, 
    \item $\mathsf{T}(\neg \varphi) = \Box \neg \mathsf{T}(\varphi)$, 
    \item $\mathsf{T}(\varphi \to \psi) = \Box (\mathsf{T}(\varphi) \to \mathsf{T}(\psi))$.  
\end{itemize}
\end{defn}

For any intermediate logic $L$ and normal modal logic $M$ extending $\SF$, we say that $M$ is a \emph{modal companion} of $L$ iff for any formula $\varphi$ of intermediate logic, $L \vdash \varphi \iff M \vdash \mathsf{T}(\varphi)$.

\begin{fact}[Maksimova \cite{Maksimova82}]\label{fact:LIP}
For any intermediate propositional logic $L$, if some modal companion of $L$ has LIP, then $L$ also has LIP. 
\end{fact}

For example, LIP for $\Cl$, $\LS$, $\LP_2$, $\KC$, and $\mathbf{Int}$ follow from LIP for $\SFi$, $\mathbf{GW.2}$, $\mathbf{GW}$, $\mathbf{S4.2}$, and $\SF$, respectively. 
Since all modal companions of $\LC$ do not have CIP \cite{Maksimova82b}, Kuznets and Lellmann \cite{KL18,KL21} directly proves LIP for $\LC$ without using Fact \ref{fact:LIP}. 
Among the intermediate logics having CIP, only the logic $\LV$ was not known to have LIP or not (cf.~\cite{Maksimova14}). 
Since we proved that a modal companion $\GV$ of $\LV$ has LIP (Theorem \ref{thm:GV}), we obtain the following corollary, which yields the equivalence $(1 \Leftrightarrow 3)$ of Theorem \ref{thm:intermediate}. 

\begin{cor}\label{cor1}
    $\LV$ has LIP. 
\end{cor}

Next, we consider ULIP. 
We say that a logic $L$ is \emph{locally tabular} iff for any finite set $P$ of propositional variables, there are finitely many formulas built from variables in $P$ up to $L$-provable equivalence. 
For example, it is known that $\Cl$, $\LS$, $\LV$, $\LP_2$, and $\LC$ are locally tabular. 
For locally tabular logics, LIP and ULIP coincide. 

\begin{fact}[Kurahashi {\cite[Proposition 3]{Kurahashi20}}]
If a locally tabular logic $L$ has LIP, then $L$ also has ULIP. 
\end{fact}

So, we obtain the following corollary. 

\begin{cor}
    The logics $\Cl$, $\LS$, $\LV$, $\LP_2$, and $\LC$ have ULIP. 
\end{cor}

Therefore, we only need to consider ULIP for $\mathbf{Int}$ and $\mathbf{KC}$. 
Here, we prove an analogue of Fact \ref{fact:LIP}. 

\begin{prop}\label{prop:ULIP1}
For any intermediate propositional logic $L$, if some modal companion of $L$ has ULIP, then $L$ also has ULIP. 
\end{prop}
\begin{proof}
Suppose that a modal companion $M$ of $L$ has ULIP. 
Let $\varphi$ be any formula and $P^+, P^-$ be any finite sets of propositional variables. 
Let $\theta$ be a uniform Lyndon interpolant of $(\mathsf{T}(\varphi), P^+, P^-)$ in $M$. 
Let $\mathsf{s}$ be the uniform substitution such that $\mathsf{s}(p) = \Box p$. 
It is easy to show that there exists a formula $\xi$ of intermediate logic satisfying the following conditions (cf.~\cite[Theorem 14.9]{CZ97}): 
\begin{enumerate}
    \item $\SF \vdash \Box \mathsf{s}(\theta) \leftrightarrow \mathsf{T}(\xi)$, 
    \item $v^\circ(\xi) = v^\circ(\theta)$ for $\circ \in \{+, -\}$. 
\end{enumerate}

We show that $\xi$ is a uniform Lyndon interpolant of $(\varphi, P^+, P^-)$ in $L$. 

\medskip

1. Since $M \vdash \mathsf{T}(\varphi) \to \theta$ and $\SF \vdash \mathsf{T}(\varphi) \leftrightarrow \Box \mathsf{T}(\varphi)$, we have $M \vdash \mathsf{T}(\varphi) \to \Box \theta$. 
Since $M$ is closed under applying uniform substitutions and $\SF \vdash \mathsf{T}(\varphi) \leftrightarrow \mathsf{s} \left(\mathsf{T}(\varphi) \right)$, we obtain $M \vdash \mathsf{T}(\varphi) \to \Box \mathsf{s}(\theta)$. 
Hence, $M \vdash \mathsf{T}(\varphi) \to \mathsf{T}(\xi)$, and so $M \vdash \mathsf{T}(\varphi \to \xi)$. 
Since $M$ is a modal companion of $L$, we get $L \vdash \varphi \to \xi$. 

\medskip

2. For $\circ\in \{+, -\}$, we have $v^\circ(\xi) = v^\circ(\theta) \subseteq v^\circ\left(\mathsf{T}(\varphi) \right) \setminus P^\circ = v^\circ(\varphi) \setminus P^\circ$. 

\medskip

3. Let $\psi$ be any formula such that $L \vdash \varphi \to \psi$ and $v^\circ(\psi) \cap P^\circ = \emptyset$. 
Then, $M \vdash \mathsf{T}(\varphi) \to \mathsf{T}(\psi)$. 
Since $v^\circ\left(\mathsf{T}(\psi) \right) \cap P^\circ = v^\circ(\psi) \cap P^\circ = \emptyset$, we obtain $M \vdash \theta \to \mathsf{T}(\psi)$. 
Then, in the similar way as above, we obtain $M \vdash \mathsf{T}(\xi \to \psi)$. 
Therefore, $L \vdash \xi \to \psi$. 
\end{proof}

It was proved in \cite{Kurahashi20} that a modal companion $\mathbf{Grz}$ of $\mathbf{Int}$ has ULIP. 
So, we immediately obtain the following corollary. 

\begin{cor}\label{cor2}
    $\mathbf{Int}$ has ULIP. 
\end{cor}

Finally, we focus on ULIP for $\mathbf{KC}$. 
For this, it suffices to prove ULIP for a modal companion $\mathbf{Grz.2}$ of $\mathbf{KC}$. 
Let $\alpha(p)$ be the formula $\Box (\Box \Diamond p \leftrightarrow \Diamond \Box p)$, then $\mathbf{Grz.2}$ is the least normal extension of $\mathbf{Grz}$ containing $\alpha(p)$. 
We use the following facts. 

\begin{fact}[Cf.~{\cite[Lemma 5.32 and Proposition 5.35]{GM05}}]\label{fact:conservation}
For any formula $\varphi$ with $v(\varphi) = \{q_0, \ldots, q_{k-1}\}$, if $\mathbf{Grz.2} \vdash \varphi$, then
\[
    \mathbf{Grz} \vdash \alpha(\bot) \land \alpha(q_0) \land \cdots \land \alpha(q_{k-1}) \to \varphi.
\]
\end{fact}

\begin{fact}[Cf.~{\cite[Lemma 5.33]{GM05}}]\label{lem:subst}
Let $L$ be a normal modal logic and $\varphi(p), \psi$ and $\rho$ be any formulas such that $L \vdash \psi \to \rho$.  
\begin{enumerate}
    \item If $p \notin v^-(\varphi)$, then $L \vdash \varphi(\psi) \to \varphi(\rho)$. 
    \item If $p \notin v^+(\varphi)$, then $L \vdash \varphi(\rho) \to \varphi(\psi)$. 
\end{enumerate}
\end{fact}

\begin{thm}\label{thm:Grz.2}
$\mathbf{Grz.2}$ has ULIP. 
\end{thm}
\begin{proof}
Let $\varphi$ be any formula and $P^+, P^-$ be any finite sets of propositional variables. 
If $\circ$ is $+$ (resp.~$-$), then let $\bullet$ denote $-$ (resp.~$+$). 
Let $X$ be the union of the following three sets: 
\[
    v^+(\varphi) \cap v^-(\varphi), \quad (v^+(\varphi) \setminus v^-(\varphi)) \cap P^-, \quad (v^-(\varphi) \setminus v^+(\varphi)) \cap P^+.
\]
Also let $Y$ be the union of the following two sets: 
\[
    v^+(\varphi) \setminus (v^-(\varphi) \cup P^-), \quad v^-(\varphi) \setminus (v^+(\varphi) \cup P^+).
\]
Then, $v(\varphi)$ is the disjoint union of $X$ and $Y$. 

By ULIP for $\mathbf{Grz}$, we find a uniform Lyndon interpolant $\theta$ of 
\[
    \left(\alpha(\bot) \land \bigwedge_{p \in X}\alpha(p) \land \varphi, P^+, P^- \right)
\]
in $\mathbf{Grz}$.  
We prove that $\theta$ is also a uniform Lyndon interpolant of $\left(\varphi, P^+, P^- \right)$ in $\mathbf{Grz.2}$.

\medskip

1. Since $\mathbf{Grz} \vdash \alpha(\bot) \land \bigwedge_{p \in X}\alpha(p) \land \varphi \to \theta$, we get $\mathbf{Grz.2} \vdash \varphi \to \theta$.

\medskip

2. We have:
\begin{align*}
     v^\circ(\theta) & \subseteq v^\circ \left(\alpha(\bot) \land \bigwedge_{p \in X}\alpha(p) \land \varphi \right) \setminus P^\circ \\
     & = \bigl[ v^\circ(\varphi) \cup \bigl\{\bigl((v^\bullet(\varphi) \setminus v^\circ(\varphi)\bigr) \cap P^\circ \bigr\} \bigr] \setminus P^\circ \\
     & = v^\circ(\varphi) \setminus P^\circ. 
\end{align*}

\medskip

3. Let $\psi$ be any formula such that $\mathbf{Grz.2} \vdash \varphi \to \psi$ and $v^\circ(\psi) \cap P^\circ = \emptyset$. 
By Fact \ref{fact:conservation}, we have
\[
    \mathbf{Grz} \vdash \alpha(\bot) \land \bigwedge_{p \in v(\varphi) \cup v(\psi)} \alpha(p) \land \varphi \to \psi. 
\]
Since $v(\varphi) \cup v(\psi)$ is the disjoint union of $X$, $Y$, and $v(\psi) \setminus v(\varphi)$, we obtain
\begin{equation}\label{Grz}
    \mathbf{Grz} \vdash \alpha(\bot) \land \bigwedge_{p \in X} \alpha(p) \land \varphi \to \bigl(\bigwedge_{q \in Y \cup (v(\psi) \setminus v(\varphi))} \alpha(q) \to \psi \bigr). 
\end{equation}
Let $\mathsf{s}$ be the uniform substitution defined as follows: 
\[
    \mathsf{s}(p) : = \begin{cases}
    \top & \text{if}\ p \notin v^-(\varphi) \cup v^+(\psi), \\
    \bot & \text{else if}\ p \notin v^+(\varphi) \cup v^-(\psi), \\
    p & \text{otherwise.}
    \end{cases}
\]
Since $\K \vdash p \to \top$ and $\K \vdash \bot \to p$, by Fact \ref{lem:subst}, we have $\K \vdash \varphi \to \mathsf{s}(\varphi)$ and $\K \vdash \mathsf{s}(\psi) \to \psi$. 
Since $\mathbf{Grz} \vdash \alpha(\top)$, we have
\[
    \mathbf{Grz} \vdash \alpha(\bot) \land \bigwedge_{p \in X} \alpha(p) \land \varphi \to \mathsf{s} \left(\alpha(\bot) \land \bigwedge_{p \in X} \alpha(p) \land \varphi\right). 
\]
Let $Z$ be the intersection of the following three sets: 
\[
    Y \cup (v(\psi) \setminus v(\varphi)), \quad v^-(\varphi) \cup v^+(\psi), \quad v^+(\varphi) \cup v^-(\psi).
\]
Then, it is shown that $Z$ is included in the union of the following three sets: 
\[
    (v^+(\varphi) \cap v^+(\psi)) \setminus (v^-(\varphi) \cup P^-), \quad (v^-(\varphi) \cap v^-(\psi)) \setminus (v^+(\varphi) \cup P^+), \quad (v^+(\psi) \cap v^-(\psi)) \setminus v(\varphi).
\]
Moreover, since $v^\circ(\psi) \cap P^\circ = \emptyset$, we have $Z \subseteq (v^+(\psi) \setminus P^-) \cup (v^-(\psi) \setminus P^+)$. 
Then, by applying the uniform substitution $\mathsf{s}$ to (\ref{Grz}), it follows from the above observations that
\[
    \mathbf{Grz} \vdash \alpha(\bot) \land \bigwedge_{p \in X} \alpha(p) \land \varphi \to \left(\bigwedge_{\substack{q \in v^+(\psi) \setminus P^- \\
     \text{or}\ q \in v^-(\psi) \setminus P^+}} \alpha(q) \to \psi \right). 
\]
We have: 
\begin{align*}
    & v^\circ \left(\bigwedge_{\substack{q \in v^+(\psi) \setminus P^- \\
     \text{or}\ q \in v^-(\psi) \setminus P^+}} \alpha(q) \to \psi \right) \cap P^\circ \\
    & = \bigl[ (v^+(\psi) \setminus P^-) \cup (v^-(\psi) \setminus P^+) \cup v^\circ(\psi) \bigr] \cap P^\circ \\
    & = \bigl[(v^\bullet(\psi) \setminus P^\circ) \cup v^\circ(\psi) \bigr] \cap P^\circ \\
    & = v^\circ(\psi) \cap P^\circ \\
    & = \emptyset. 
\end{align*}
Therefore, we obtain
\[
    \mathbf{Grz} \vdash \theta \to \left(\bigwedge_{\substack{q \in v^+(\psi) \setminus P^- \\ \text{or}\ q \in v^-(\psi) \setminus P^+}} \alpha(q) \to \psi \right).
\]
We conclude $\mathbf{Grz.2} \vdash \theta \to \psi$. 
\end{proof}

\begin{cor}\label{cor3}
$\mathbf{KC}$ has ULIP. 
\end{cor}

We have proved the equivalence $(1 \Leftrightarrow 4)$ of Theorem \ref{thm:intermediate}. 

\section*{Acknowledgements}

This work was supported by JSPS KAKENHI Grant Number JP23K03200.

\bibliographystyle{plain}
\bibliography{ref}

\end{document}